\documentclass[11pt,reqno]{amsart}
\usepackage{amsmath}
\usepackage{etoolbox}
\usepackage{amscd, amsfonts, amssymb, graphicx, color}
\usepackage{url}
\usepackage{tikz-cd}
\usepackage{hyphenat}
\usepackage{mathtools}
\usepackage{cite}
\usepackage[colorlinks=true]{hyperref}
\usepackage{titlesec}
\usepackage{geometry}
\usepackage{hyphenat}
\usepackage{float}
\titleformat{\section}
{\normalfont\bfseries\large\centering}{\thesection}{1em}{}
\makeatletter
\patchcmd\maketitle
{\uppercasenonmath\shorttitle}
{}
{}{}
\patchcmd\maketitle
{\@nx\MakeUppercase{\the\toks@}}
{\the\toks@}
{}
{}{}
\patchcmd\@settitle{\uppercasenonmath\@title}{\Large}{}{}
\patchcmd\@setauthors
{\MakeUppercase{\authors}}
{\authors}
{}{}
\makeatother
\newtheorem{theorem}{Theorem}[section]

\newtheorem{lemma}{Lemma}[section]

\newtheorem{remark}{Remark}[section]

\newtheorem{corollary}{Corollary}[section]
\newtheorem{Proof of Theorem}{Proof}
\newtheorem{proposition}[theorem]{Proposition}

\hypersetup{urlcolor=blue, citecolor=red, linkcolor= blue}

\newcommand{\norm}[1]{\left\lVert#1\right\rVert}

\makeatletter
\renewcommand\subsection{\@startsection{subsection}{2}%
	\z@{.7\linespacing\@plus\linespacing}{.5\linespacing}%
	{\normalfont\bfseries}}
\makeatother

\begin{document}
	\title[$q$-Numerical Radius of Sectorial Matrices and Operator Matrices]{$q$-Numerical radius of sectorial matrices and $2 \times 2$ operator matrices}

	\author{Jyoti Rani}
	\address{Department of mathematics, Indian Institute of Technology Bhilai, Durg, India 491002}
	\email{jyotir@iitbhilai.ac.in}
	
	\author{Arnab Patra}
	\address{Department of mathematics, Indian Institute of Technology Bhilai, Durg, India 491002}
	\email{arnabp@iitbhilai.ac.in}

	\subjclass[2020]{ 15A60; 15B48; 47B44; 47A63}
	
	\keywords{$q$-numerical radius, sectorial matrices, operator matrices, normaloid operator.}
	
	\begin{abstract} 
		This article focuses on several significant bounds of $q$-numerical radius $w_q(A)$ for sectorial matrix $A$ which refine and generalize previously established bounds. One of the significant bounds we have derived is as follows:
		\[\frac{|q|^2\cos^2\alpha}{2}	\|A^*A+AA^*\|  \le w_q^2(A)\le \frac{\left(\sqrt{(1-|q|^2)\left(1+2sin^2(\alpha)\right)}+ |q|\right)^2}{2} \|A^*A+AA^*\|,\]
		where $ A $ is a sectorial matrix. Also, upper bounds for commutator and anti-commutator matrices and relations between $w_q(A^t)$ and $w_q^t(A)$ for non-integral power $t\in [0,1]$ are also obtained. Moreover, a few significant estimations of $q$-numerical radius of off-diagonal $2\times2$ operator matrices are developed.   
		
	\end{abstract}
	
	\maketitle
	
	\section{Introduction}
	Let $\mathcal{B(H)}$ be the $C^*$ algebra of all bounded linear operators acting on the Hilbert space $(\mathcal{H}, \langle .,. \rangle )$ equipped with the operator norm. For any $T \in \mathcal{B(H)}$, numerical range $W(T)$, numerical radius $w(T)$, and operator norm $\|T\|$ are defined, respectively, by 
	\begin{equation*}
		W(T)= \{\langle Tx,x \rangle  : x \in \mathcal{H}, \|x\|=1\},
	\end{equation*} 
	\begin{equation*}
		w(T)= \sup\{| \langle Tx,x \rangle | : x \in \mathcal{H}, \|x\|=1\}, \mbox{and}
	\end{equation*}
	\begin{equation*}
		\|T\|= \sup\{| \langle Tx,y \rangle | : x,y \in \mathcal{H}, \|x\|=\|y\|=1\}.
	\end{equation*}	
	It is well known that $w(.)$ defines a norm on $\mathcal{B(H)}$, which is equivalent to the usual operator norm $\|T\|$. In fact, for every $T \in \mathcal{B(H)}$ the following relation holds,
	\begin{equation}\label{eq2.1}
		\frac{\|T\|}{2} \le w(T) \le \|T\|.
	\end{equation}
	
	The left-hand side inequality is transformed into an equality when $T^2=0$ and the other inequality becomes an equality when the operator $T$ is normal.

	
	
	In 2005, Kittaneh \cite{kittaneh2005numerical} obtained one more refinement of inequality (\ref{eq2.1}) as follows		\begin{equation}\label{eq3.2}
		\frac{1}{4}\|T^*T+TT^*\| \le w^2(T) \le \frac{1}{2}\|T^*T+TT^*\|.
	\end{equation}
	For more such inequalities one may refer to the recent articles
	\cite{bhunia2021development,moslehian2020seminorm, bhunia2019numerical,najafi2020some,feki2022some,feki2022some1} along with the references therein 
	and the books \cite{bhunia2022lectures,gau2021numerical}.

	Recently studies have been made on the numerical radius of a particular class of matrices, known as sectorial matrices. Let $M_n$ be the algebra of $n \times n$ matrices. Every $A \in M_n$ admits the decomposition $A= \mathcal{R}(A)+i \mathcal{I}(A)$, where $\mathcal{R}(A)=\frac{A+A^*}{2}$ and $\mathcal{I}(A)=\frac{A-A^*}{2i}$ are hermition matrices. A matrix $A\in M_n$ is said to be accretive if $\mathcal{R}(A)>0.$ Also, $A\in M_n$ is said to be accretive-dissipative if $\mathcal{R}(A)>0$ and $\mathcal{I}(A)>0$. In other words, $A$ is accretive if the numerical range $W(A)$ is a subset of the right-half plane. If the numerical range $W(A)$ is a subset of a sector $S_\alpha$ for some $\alpha \in \mathclose{[}0,\frac{\pi}{2}\mathopen{)}$ in the right half of the complex plane is said to be sectorial where
	\begin{equation*}
		S_{\alpha}= \{ z \in \mathbb{C} : \mathcal{R} z >0 , |\mathcal{I}z| \le \tan(\alpha)(\mathcal{R} z)\}.
	\end{equation*}
	The class of all $n \times n$ sectorial matrices where $W(T) \subseteq S_{\alpha}$ is denoted by $\prod_{s,\alpha}^n$. If $\alpha=0$, then the sector $S_{\alpha}$, reduces to the interval $(0, \infty)$ which reduces the class $\prod_{s,\alpha}^n$ to the set of all positive matrices in $M_n$. The numerical range and radius of sectorial matrices have been explored by several authors. In particular, Samah Abu Sammour et al. \cite{sammour2022geometric}, Yassine Bedrani et al. \cite{bedrani2021numerical}, and Pintu Bhunia et al. \cite{bhunia2023numerical} has focused their study on the bounds of the numerical radius of sectorial matrices. However, we explore the concept of a much generalized numerical range, namely, the $q$-numerical range of sectorial matrices.  The $q$-numerical range of $T \in \mathcal{B}(H)$ is defined by,

	\begin{equation*}
		W_q(T)=\{ \langle Tx,y \rangle: x,y \in H, \|x\|=\|y\|=1, \langle x,y \rangle=q \},
	\end{equation*}
	where $|q| \le 1$.
	The $q$-numerical radius $w_q(T)$ of $T \in \mathcal{B(H)}$ is 
	\begin{equation*}
		w_q(T)=\sup_{w \in W_q(T)}|w|.
	\end{equation*}
	The following relations can be easily derived:
	\begin{equation}\label{ri}
		w_q(\mathcal{R}(T)) \le w_q(T),~w_q(\mathcal{I}(T)) \le w_q(T).
	\end{equation}
	Limited research work has been available in the literature on the $q$-numerical range.
	Several notable studies concerning the $q$-numerical range and radius include \cite{tsing1984constrained,li1998q,psarrakos2000q,chien2002davis}.
	For a detailed review, one can refer to the book \cite[p.380]{gau2021numerical}. Recently in \cite{fakhri2024q}, for $T \in \mathcal{B(H)}$ and $q \in (0,1)$, a few significant estimations of $q$-numerical radius are provided such as
	\begin{equation}\label{q1}
		\frac{q}{2(2 - q^2)} \|T\| \leq w_q(T) \leq \|T\|,  
	\end{equation}
	\begin{equation}\label{q2}
		\frac{1}{4} \left( \frac{q}{2 - q^2} \right)^2 \| T^*T + TT^* \| 
		\leq w_q^2(T) 
		\leq \frac{\left( q + 2\sqrt{1-q^2} \right)^2}{2} \| T^*T + TT^* \|.
	\end{equation}

	The present work is an attempt to obtain bounds of $q$-numerical radius of sectorial matrices. This leads to the refinement of several results on $q$-numerical radius. Significant results on upper bounds for the $q$-numerical radius of commutator and anti-commutator matrices, non-integral powers of matrices are obtained. Furthermore, this study also explores the $q$-numerical radius inequalities associated with $2 \times 2$ block matrices.

	
	\section{$q$-Numerical Radius Inequalities for Sectorial Matrices}

	Throughout this paper, the symbol $T$ denotes an element of $\mathcal{B(H)}$, while the letter $A$ is designated specifically for $n \times n$ matrices. We start this section by listing some known outcomes that will be required in our analysis of the principal findings.
	\begin{lemma}\label{il1.6}\cite{sammour2022geometric}
		Let $A \in \prod_{s,\alpha}^n$ for some $\alpha \in \mathclose{[}0,\frac{\pi}{2}\mathopen{)} $. Then
		\begin{equation*}
			\|\mathcal{I}(A)\| \le \sin (\alpha)w(A).
		\end{equation*}	
	\end{lemma}

	\begin{lemma}\label{it 1.3} \cite{sammour2022geometric}
		Let $A \in \prod_{s,\alpha}^n$. Then
		\begin{equation*}
			\|A\| \le \sqrt{1+2\sin^2 (\alpha)}w(A).
		\end{equation*}
	\end{lemma}
	The following Lemma represents a relation between $|||\mathcal{R}(A)|||$ and $|||A|||$, where $|||.|||$ is an unitarily invariant norm on $M_n.$
	\begin{lemma}\label{normr1}\cite{zhang2015matrix}
		Let  $ A \in \prod_{s,\alpha}^n$ and $|||.|||$ be any unitarily invariant norm on $M_n$. Then
		\begin{equation*}
			\cos({\alpha})|||A|||\le ||| \mathcal{R}(A)||| \le |||A|||. 
		\end{equation*}
	\end{lemma}
	
	
	The following result states that,  if a matrix $A$ is sectorial, then raising it to a fractional power will still preserve its sectoriality.
	\begin{lemma}\cite{drury2015principal} \label{il1.2}
		Let $0 \le \alpha <\frac{\pi}{2}$, $0 <t<1$ and $A \in M_n$ is a square matrix with $W(A) \subseteq S_{\alpha}$. Then $W(A^t)\subseteq S_{t \alpha}$.
	\end{lemma}	 
	Additionally, it should be noted that \(W(A^{-t})\) is a subset of \(S_{t\alpha}\). This can be inferred from the fact that \(W(A^{-1})\) is a subset of \(S_{\alpha}\) when \(W(A)\) is a subset of \(S_{\alpha}\).
	
	Here are some additional important results related to the fractional powers of sectorial matrices:
	\begin{lemma}\cite{choi2019extensions}\label{il1.7}
		Let $A \in \prod_{s,\alpha}^n$ and $t \in [0,1]$. Then 
		\begin{equation*}
			\cos^{2t}(\alpha)\mathcal{R}(A^t) \le (\mathcal{R}(A))^t \le \mathcal{R}(A^t).
		\end{equation*}
	\end{lemma}
	\begin{lemma}\cite{choi2019extensions}\label{il1.9}
		Let $A \in \prod_{s,\alpha}^n$ and $t \in [-1,0]$. Then 
		\begin{equation*}
			\mathcal{R}(A^t) \le (\mathcal{R}(A))^t \le \cos^{2t}(\alpha) \mathcal{R}(A^t).
		\end{equation*}
	\end{lemma}

	Let $\mathcal{D}$ be the closed unit disc in the complex plane and $\mathcal{D'}=\mathcal{D}\setminus \{0\}.$
	An operator $T\in \mathcal{B(H)}$ is said to be normaloid if $w(T)=\|T\|$. Also, all normal operators are satisfied the similar relation. However similar equality does not hold for $q$-numerical radius.  For example, let $T=\begin{bmatrix}
		{5} & {0}\\
		{0} & {4}
	\end{bmatrix}$. Then $\|T\|=5$ and the $q$-numerical radius $w_q(T)=\frac{9|q|+1}{2} \ne \|T\|$ for all $q \in \mathcal{D}$ except $|q|=1$, follows from Theorem 3.4 \cite[p.384]{gau2021numerical}. In this context, the next theorem provides a relation between $w_q(T)$ and $\|T\|$ for normaloid operator $T$. It is noteworthy that for a normaloid operator $T$, the equality
	$r(T)=\|T\|=w(T)$ holds, where $r(T)$ is the spectral radius of $T$. 
	
	We start our main result with a relation between $w_q(T)$ and $w(T)$ for normaloid operators by using the spectrum inclusion relation.
	\begin{theorem}\label{t1.15}
		If $T \in\mathcal{B(H)}$ is a normaloid operator and $q \in \mathcal{D}$, then 
		\begin{equation*}
			|q| w(T) \le w_q(T) \le w(T).
		\end{equation*}
	\end{theorem}
	\begin{proof}
		From Proposition 3.1 (h) \cite[p.380]{gau2021numerical}, we have the following inclusion relation	
		$$q \sigma (T) \subseteq \overline{W_q(T)} ~~~~\text{for}~|q| \le 1.$$
		It follows that
		\begin{equation*}
			|q|r(T) \le w_q(T).
		\end{equation*}
		As $T$ is normaloid, we have  
		\begin{equation*}
			|q|w(T) \le w_q(T).
		\end{equation*}
		Also, it is well-known that $w_q(T) \le \|T\|$ for all $T \in \mathcal{B(H)}$. Since $T$ is normaloid, we have
		\begin{equation*}
			w_q(T) \le \|T\|=w(T).
		\end{equation*}
		This completes the proof.
	\end{proof}
	
	\begin{remark} Few important observations are mentioned below.
		\begin{enumerate}
			
			\item[(i)] For normal $T,$ the relation $w(T) = \|T\|$ holds and subsequently we have $|q| \|T\| \leq w_q(T) \leq \|T\|$ which is recently obtained in \cite[Theorem 1.4]{stankovic2024some}.	
			\item[(ii)] For $q \in (0,1)$, $q \ge \frac{q}{2-q^2}$, thus Theorem \ref{t1.15} is a refinement of \cite[Theorem 2.1]{fakhri2024q} for normal operator $T$.  
			
			\item[(iii)] Theorem \ref{t1.15} is the $q$-numerical radius analouge of the well-known relation \eqref{eq2.1} for normal operators which reduces to the equality $w(T)=\|T\|$ when $q=1.$
			\item [(iv)] Both the inequalities of Theorem \ref{t1.15} are best possible.
			Let $S_1$ be the set of all $n \times n(n \ge 2)$ real scalar matrices. Let $ S_1 \ni A_1=\eta I$, where $\eta(\ne 0) \in \mathbb{R}$ and $I$ is $n \times n$ identity matrix. Therefore, $w_q(A_1) = |q|w(A_1)$. Let $S$ be the right shift operator. Then $w_q(S)=w(S)=\mathcal{D}$ \cite[p.384]{gau2021numerical}.
		\end{enumerate}
	\end{remark}
	For any $T \in \mathcal{B(H)}$, we have $T= \mathcal{R}(T)+i\mathcal{I}(T)$, where $\mathcal{R}(T)=\frac{T+T^*}{2}$ and $\mathcal{I}(T)=\frac{T-T^*}{2i}$. Also, $\mathcal{R}(T)$ and $\mathcal{I}(T)$ are self-adjoint operators. An easy calculation leads to the following corollary. 
	\begin{corollary}\label{co1.6}
		If $T \in \mathcal{B(H)}$ and $q \in \mathcal{D}$, then we have 
		\begin{itemize}
			\item [(a)] $|q|\|\mathcal{R}(T)\| \le w_q(\mathcal{R}(T))\le \|\mathcal{R}(T)\|$,
			\item [(b)]  $|q|\|\mathcal{I}(T)\| \le w_q(\mathcal{I}(T))\le \|\mathcal{I}(T)\|.$
		\end{itemize}
	\end{corollary}

	The next result provides the equivalence of two norms $\|A\|$ and $w_q(A)$ and extends the inequality (\ref{eq2.1}) for $q$-numerical radius of sectorial matrices.
	\begin{theorem}\label{t1.16}
		\begin{itemize}
			\item [(a)]If $A \in \prod_{s,\alpha}^n$ and $q \in \mathcal{D'}$ then 
			\begin{equation*}
				|q| \cos{\alpha} \|A\| \le w_q(A) \le \|A\|.
			\end{equation*}
			\item[(b)]    Let $q \in \mathcal{D'}$. If either $A, A^2 \in \prod_{s,\alpha}^n$ or $A$ is accerative-dissipative, then
			\begin{equation*}\label{sqad}
				\frac{|q|}{\sqrt{2}}\|A\| \le w_q(A) \le \|A\|.
			\end{equation*}  
		\end{itemize}
	\end{theorem}
	\begin{proof}
		\begin{itemize}
			\item [(a)] 	Using Lemma \ref{normr1},  Corollary \ref{co1.6}, and relation (\ref{ri}) we have
			\begin{equation*}
				\cos(\alpha) \|A\| \le \|\mathcal{R}(A)\|\le  \frac{1}{|q|}w_q(\mathcal{R}(A)) \le \frac{1}{|q|}w_q(A) \le  \frac{1}{|q|} \|A\|.
			\end{equation*}
			Hence, the required result holds.   
			\item[(b)]    If $A^2 \in \prod_{s,\alpha}^n$ then Lemma \ref{il1.2} implies that
			\begin{equation}\label{square}
				W(A)=W((A^2)^\frac{1}{2}) \subset S_\frac{\alpha}{2}\subset S_\frac{\pi}{4}.    
			\end{equation}
			Also, if $A$ is accerative-dissipative, then
			\begin{equation}\label{accdis}
				W(e^{\frac{-i\pi}{4}}A)\subset S_\frac{\pi}{4}. 
			\end{equation}
			Thus, Theorem \ref{t1.16} and relations (\ref{square}) and (\ref{accdis}) all together give us the desired result.  
		\end{itemize}
	\end{proof}
	\begin{remark}
		\begin{itemize}
			\item [(i)]If $\alpha \in \left[0,\frac{\pi}{3}\right)$ and $q=1$, then the lower bound mentioned in Theorem \ref{t1.16}(a) is a refinement of lower bound mentioned in (\ref{eq2.1}). Moreover, for $\cos(\alpha) \ge \frac{1}{2(2-q^2)}$, $q \in (0,1)$, the lower bound mentioned in Theorem \ref{t1.16}(a) is a refinement of the lower bound mentioned in relation \eqref{q1}.
			\item[(ii)]  Theorem \ref{t1.16}(b) provides a more precise version of Theorem \ref{t1.16} for $\alpha \in \left(\frac{\pi}{4},\frac{\pi}{2}\right)$.
		\end{itemize}
		
	\end{remark}
	
	The following corollary provides an upper bound for \(w_q(AB)\).
	Recall that for any two matrices $A,B \in M_n$, we have \cite[p.125]{gau2021numerical}
	\begin{equation}\label{4in}
		w(AB)\le 4w(A)w(B). 
	\end{equation} 
	
	Also, it is well-known that if $A$ and $B$ are positive matrices then
	\begin{equation}\label{positive}
		w(AB)\le w(A)w(B).
	\end{equation}
	
	However, the relations \eqref{4in} and (\ref{positive}) do not hold for $q$-numerical radius. Consider normal matrices $A=\begin{bmatrix}
		1 &0 \\
		0 & 2
	\end{bmatrix}$ and $B=\begin{bmatrix}
		2 &0 \\
		0 & 2
	\end{bmatrix}$.
	Then by Theorem 3.4 \cite[p.384]{gau2021numerical}, we have $w_q(A)=\frac{3|q|+1}{2}$, $w_q(B)=2|q|$ and $w_q(AB)=3|q|+1$. One can check that $w_q(AB) \not \le w_q(A)w_q(B)$ for all $q \in \mathcal{D'}$ except $|q|=1$, and $w_q(AB) \not\le 4w_q(A)w_q(B)$ for $|q|< 0.25$. Our next theorem is an attempt to overcome such situations. Using Theorem \ref{t1.16}(a) and $w_q(AB) \le \|A\|\|B\|$,  we can derive the following result easily. 
	\begin{corollary}\label{th1.13}
		For $q \in \mathcal{D'}$, the following results hold true,
		\begin{itemize}
			\item [(a)] 	If $A$ and $B$ both are sectorial matrices then  
			\begin{equation*}
				|q|^2w_q(AB) \le \sec^2(\alpha)w_q(A)w_q(B).
			\end{equation*}
			\item [(b)] If $A$ and $B$ both are positive matrices then  
			\begin{equation*}\label{nee1}
				|q|^2w_q(AB) \le w_q(A)w_q(B).
			\end{equation*}
		\end{itemize}
	\end{corollary}
		
	\begin{remark}
		If we take $q=1$, Corollary \ref{th1.13}(a) implies that $	w(AB) \le\sec^2(\alpha)w(A)w(B)$ when $A$ and $B$ both are sectorial matrices, which is a refinement of inequality (\ref{4in}) when $\alpha \in \left( 0, \frac{\pi}{3} \right)$. If we take $q=1$ in Corollary \ref{th1.13}(b), we obtain the well-known result $w(AB)\le w(A)w(B)$. 
	\end{remark}

	Before proceeding to the next theorem, we establish the following construction.
	Consider $q \in \mathcal{D'}$ but $|q|\ne 1$ and $x \in \mathcal{H}$ with $\|x\|=1$. For any $z \in \mathcal{H}$ satisfying $\langle x,z \rangle=0$ and $\|z\|=1$, let $y=\overline{q}x+\sqrt{1-|q|^2}z$. This construction ensures $\|y\|=1$ and $\langle x,y \rangle=q$. Conversely, for any $y \in \mathcal{H}$ with $\|y\|=1$ and $\langle x,y \rangle=q$, set $z=\frac{1}{\sqrt{1-|q|^2}}(y-\overline{q}x)$, resulting in $\langle x,z \rangle=0$ and $\|z\|=1$. Thus, there exists a one-to-one correspondence between such a $z$ and $y$.
	
	Our next focus is to establish a relation between the $q$-numerical radius and classical numerical radius for sectorial matrices. It is evident that if $A$ is normaloid then $w_q(A)\le w(A)$. 
	Take $q \in (0,1)$ and a non-normaloid matrix $A= \begin{bmatrix}
		0 & 1 \\
		0 & 0
	\end{bmatrix}$. Then for $\begin{bmatrix}
		x_1 \\
		x_2
	\end{bmatrix}, \begin{bmatrix}
		y_1 \\
		y_2
	\end{bmatrix}\in\mathbb{C}^2$, with $\|x\|=\|y\|=1$ and $\langle x,y \rangle =q,$ we have $\langle Ax,y \rangle=x_2\overline{y_1}$.
	Thus 
	\begin{equation*}
		w_q(A)=\sup_{
			\|x\|=1,\\ 
			\|y\|=1\\, 
			\langle x,y \rangle=q 
		}|x_2\overline{y_1}|.
	\end{equation*}
	If we take $x=\left(\frac{1}{\sqrt{2}},\frac{1}{\sqrt{2}}\right)$ and $y=\left(\frac{{q}-\sqrt{1-q^2}}{\sqrt{2}}, \frac{{q}+\sqrt{1-q^2}}{\sqrt{2}}\right)$, then 
	$w_q(A) \ge \frac{q+\sqrt{1-q^2}}{{2}}$. Also, $w(A) =\frac{1}{2}$. In this case $w_q(A) > w(A),~ q\in (0,1)$. In general, there is no relation between classical numerical range and $q$- numerical range. Our next result is an attempt in this direction for sectorial matrices. 
	\begin{theorem}\label{thmrelation}
		Let $ A \in \prod_{s,\alpha}^n$ and $q \in \mathcal{D}$. Then 
		\begin{itemize}
			
			\item [(a)]$
			w_q(A) \le \left(\sqrt{(1-|q|^2)\left(1+2\sin^2(\alpha)\right)}+ |q|\right) w(A),
			$
			\item [(b)]$
			w_q(A) \le \sqrt{1+\sin^2(\alpha)}w(A).
			$
		\end{itemize}	
	\end{theorem}
	\begin{proof} 
		\begin{itemize}
			
			\item[(a)] 
			When $|q|=1$, using \cite[Proposition 3.1(f)]{gau2021numerical}, we have $w_q(A)=w(A)$.
			Let $x,y \in \mathcal{H}$ such that $\|x\|=\|y\|=1$ and $\langle x,y \rangle=q$. Then, using the construction described above, $y$ can be expressed as $\overline{q}x+\sqrt{1-|q|^2}z$, where $z \in \mathcal{H}$, $\|z\|=1$ and $\langle x,z \rangle =0$. Therefore,
			\begin{align*}
				|\langle Ax,y \rangle| \le & |\langle Ax,\overline{q}x+\sqrt{1-|q|^2}z \rangle|\\
				\le & |q||\langle Ax,x \rangle|+\sqrt{1-|q|^2}|\langle Ax, z\rangle|\\
				\le & |q||\langle Ax,x \rangle|+\sqrt{1-|q|^2}|\|A\|. 
			\end{align*}
			Using Lemma \ref{it 1.3}, we get
			\begin{eqnarray*}
				| \langle Ax,y \rangle| \le \sqrt{(1-|q|^2)(1+2\sin^2(\alpha))}w(A)+ |q| |\langle Ax,x \rangle|.
			\end{eqnarray*}
			This implies,
			\begin{equation*}
				| \langle Ax,y \rangle| \le \left(\sqrt{(1-|q|^2)(1+2\sin^2(\alpha))}+ |q| \right)w(A).
			\end{equation*}
			Taking supremum over all $x,y$ with $\|x\|=\|y\|=1$ and $\langle x,y \rangle=q$, we have 
			\begin{equation*}
				w_q(A) \le \left(\sqrt{(1-|q|^2)(1+2\sin^2(\alpha))}+ |q| \right)w(A).
			\end{equation*}
			\item [(b)]Using the cartesian decomposition $A =  \mathcal{R}(A) + i  \mathcal{I}(A),$ we obtain
			\begin{align*}
				| \langle Ax,y \rangle| \le & \sqrt{\langle \mathcal{R}(A)x,y \rangle^2 +\langle \mathcal{I}(A)x,y \rangle^2 }\\
				\le & \left( \sqrt{\|\mathcal{R}(A)\|^2+\|\mathcal{I}(A)\|^2}\right) \|x\|\|y\|.
			\end{align*}
			Using $\|\mathcal{R}(A)\| \le w(A)$ and Lemma \ref{il1.6}, we have
			\begin{equation*}
				| \langle Ax,y \rangle \le  \left( \sqrt{w^2(A)+\sin^2(\alpha)w^2(A)}\right) \|x\|\|y\|.
			\end{equation*}
			Taking supremum over all $x,y$ on both sides with $\|x\|=\|y\|=1$ and $\langle x,y \rangle=q$, we have 
			\begin{equation*}
				w_q(A) \le \sqrt{1+\sin^2(\alpha)}w(A).
			\end{equation*}
		\end{itemize}  
	\end{proof}
	\begin{remark}
		Theorem \ref{thmrelation} enables us to explore the ratio $\frac{w_q(A)}{w(A)}$ in terms of $q$ and the sectorial index $\alpha$. Also, for different values of $\alpha$ $(\text{say}~ 0, \frac{\pi}{4}, \frac{\pi}{2})$ we can compare the bounds in part $(a)$ and part $(b)$ as follows:
		\begin{itemize}
			\item [(i)] For $\alpha=\frac{\pi}{2}$, $\sqrt{2}=\sqrt{1+\sin^2(\alpha)} \le q+ \sqrt{3(1-q^2)}=\sqrt{(1-|q|^2)\left(1+2\sin^2(\alpha)\right)}+ |q|$ holds when $ q \in [0, 0.97]$,

			
			\item [(ii)] For $\alpha=\frac{\pi}{4}$, $\sqrt{\frac{3}{2}}=\sqrt{1+\sin^2(\alpha)} \le q+ \sqrt{2(1-q^2)}=\sqrt{(1-|q|^2)\left(1+2\sin^2(\alpha)\right)}+ |q|$ holds when $ q \in [0,0.98]$,
			
			\item [(iii)] For $\alpha=0$, $1=\sqrt{1+\sin^2(\alpha)} \le q+ \sqrt{1-q^2}=\sqrt{(1-|q|^2)\left(1+2\sin^2(\alpha)\right)}+ |q|$ holds when $ q \in [0,1)$.
			\begin{figure}[!htb]
				\begin{minipage}[b]{0.32\textwidth}
					
					\includegraphics[width=\linewidth]{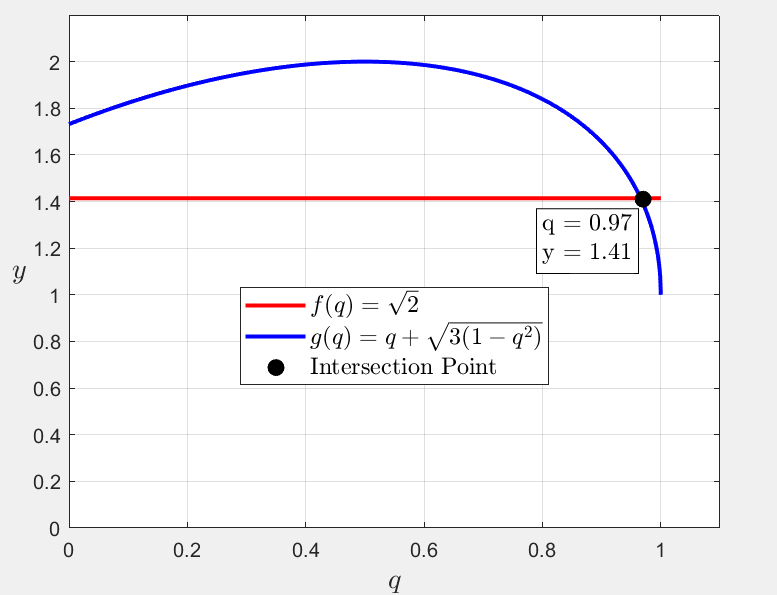}
					\begin{small}
						\begin{equation*}
							\alpha = \pi/2
						\end{equation*}
					\end{small}
				\end{minipage}\hfill
				\begin{minipage}[b]{0.32\textwidth}
					
					\includegraphics[width=\linewidth]{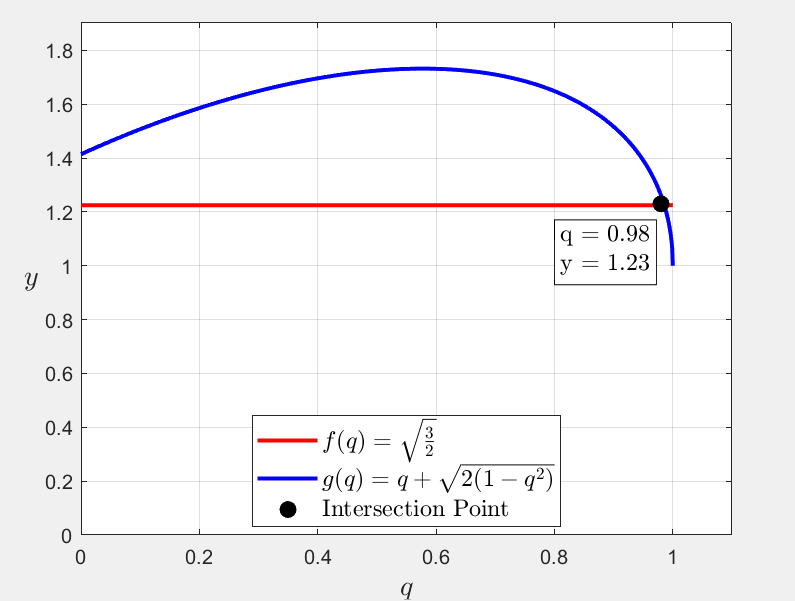}
					\begin{small}
						\begin{equation*}
							\alpha = \pi/4
						\end{equation*} 
					\end{small}
				\end{minipage}\hfill
				\begin{minipage}[b]{0.32\textwidth}
					\raisebox{+0.0262\height}{\includegraphics[width=\linewidth]{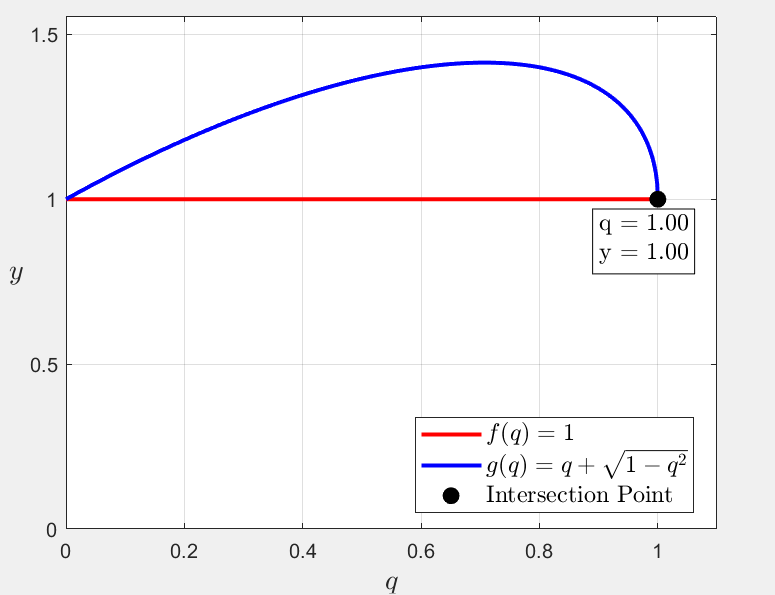}}
					\begin{small}
						\begin{equation*}
							\alpha = 0
						\end{equation*} 
					\end{small}
				\end{minipage}
			\end{figure}

			
		\end{itemize}
		The analysis of the figures demonstrates that the bound presented in part $(b)$ is a refinement of the bound presented in part $(a)$ within a specific range of $q$.
		As the value of $\alpha$ decreases, the region of $q$ where refinement occurs expands correspondingly. This observation highlights the relationship between the parameter $\alpha$ and the extent of refinement achieved in the specified region of $q$.
			

	\end{remark}
	Now we are ready to give $q$-numerical radius version of Theorem 1 \cite{kittaneh2005numerical} for sectorial matrices. Furthermore, this result serves as a refinement of Theorem 3.1 in \cite{fakhri2024q}.
	\begin{theorem}\label{lt}
		Let $ A \in \prod_{s,\alpha}^n$ and $q \in \mathcal{D'}$. Then 
		\begin{equation*}\label{eq1.16}
			\frac{|q|^2\cos^2\alpha}{2}	\|A^*A+AA^*\|  \le w_q^2(A)\le \left(\sqrt{(1-|q|^2)\left(1+2sin^2(\alpha)\right)}+ |q|\right)^2 \frac{\|A^*A+AA^*\|}{2}.
		\end{equation*}
	\end{theorem}
	\begin{proof}
		From Lemma \ref{il1.6}, we have  
		\begin{equation*}\label{eq1.8}
			\|\mathcal{I}(A)\| \le \sin (\alpha) w(A) \le \sin(\alpha)\|A\| .
		\end{equation*}
		Using Theorem $\ref{t1.16}(a)$, we have 
		\begin{equation}\label{eq1.18}
			\|\mathcal{I}(A)\| \le \frac{1}{|q|}\tan (\alpha) w_q(A).
		\end{equation}
		Now,
		\begin{align*}
			\|A^*A+AA^*\|=&2\|\mathcal{R}^2(A)+\mathcal{I}^2(A)\|\\
			\le & 2(\|\mathcal{R}(A)\|^2+\|\mathcal{I}(A)\|^2).
		\end{align*}
		Using equation (\ref{eq1.18}) and Corollary \ref{co1.6}(a), we have 
		\begin{align*}
			\|A^*A+AA^*\| \le &2 \left(\frac{1}{|q|^2}w_q^2(\mathcal{R}(A))+\tan^2 \alpha \frac{1}{|q|^2}w_q^2(\mathcal{I}(A))\right)\\
			\le & 2 \left(\frac{1}{|q|^2}w_q^2(A)+\tan^2 \alpha \frac{1}{|q|^2}w_q^2(A)\right).
		\end{align*}
		Thus we have,
		\begin{equation*}
			\frac{|q|^2\cos^2\alpha}{2}	\|A^*A+AA^*\| \le w_q^2(A).
		\end{equation*}
		For the other part, Theorem \ref{thmrelation}(a) implies,
		\begin{equation*}
			w_q^2(A) \le \left(\sqrt{(1-|q|^2)\left(1+2\sin^2(\alpha)\right)}+ |q|\right)^2 w^2(A).   
		\end{equation*}
		Using relation \eqref{eq2.1}, we obtain
		\begin{equation*}
			w_q^2(A) \le \left(\sqrt{(1-|q|^2)\left(1+2\sin^2(\alpha)\right)}+|q|\right)^2\frac{\|A^*A+AA^*\|}{2}.  
		\end{equation*}
	\end{proof}
	
	\begin{remark}
		\begin{itemize}
			\item [(i)]If $q \in (0,1)$, the upper bound of $w_q^2(A)$ of the aforementioned theorem refines the upper bound of $w_q^2(A)$ given in relation \eqref{q2}. Furthermore, when $\cos(\alpha) \geq \frac{1}{\sqrt{2}(2-q^2)}$, the lower bound of $w_q^2(A)$ in the aforementioned theorem provides an improvement over the lower bound of $w_q^2(A)$ in relation \eqref{q2}.
			
			\item[(ii)]   If $q=1$, Theorem \ref{lt} gives us
			\begin{equation}\label{ae}
				\frac{\cos^2\alpha}{2}	\|A^*A+AA^*\|  \le w^2(A)\le \frac{\|A^*A+AA^*\|}{2}.
			\end{equation}
			Clearly, if $\alpha \in (0, \frac{\pi}{3})$ the lower bound of above inequality is a refinement of the lower bound of inequality (\ref{eq3.2}), and if $A$ is a positive matrix then inequality (\ref{ae}) gives us $w(A)=\frac{\|AA^*+A^*A\|}{2}.$  
		\end{itemize}   
	\end{remark}

	Next, we focus on the $q$-numerical radius inequalities of commutator and anti-commutator sectorial matrices.
	\begin{theorem}
		Let $B, C, D \in M_n$, $A \in \prod_{s,\alpha}^n$ and $q \in \mathcal{D'}$. Then
		\begin{equation*}
			|q|w_q(A C B \pm B D A) \leq 2\sec (\alpha)\max \{\|C\|,\|D\|\} w_q(A)\|B\| .   
		\end{equation*}
	\end{theorem}
	\begin{proof}
		As the norm satisfies the homogeneity property, we can take $\|C\|\le 1$ and $\|D\| \leq 1$. Consider,
		\[
		\begin{aligned}
			w_q(AC \pm DA) & \leq \|AC \pm DA\| \\
			& \leq 2\|A\|.
		\end{aligned}
		\]
		
		Theorem \ref{t1.16} follows that,
		\begin{equation}\label{newrel}
			w_q(AC \pm D A) \leq \frac{2}{|q|}\sec (\alpha) w_q(A).   
		\end{equation}
		If $C=D=0$ then the required result holds trivially. Let $\max \{ \|C\|, \|D\|\} \ne 0.$
		Then $\left\|\frac{C}{\max \{\|C\|,\|D\|\}}\right\| \leq 1$ and $\left\|\frac{D}{\max \{\|C\|,\|D\|\}}\right\| \leq 1$.
		
		Therefore, by replacing $C$ with $\frac{C}{\max \{\|C\|,\|D\|\}}$ and $D$ with $\frac{D}{\max \{\|C\|,\|D\|\}}$ in relation (\ref{newrel}), we have
		\begin{equation}\label{inequality}
			w_q(AC \pm D A) \leq \frac{2}{|q|}\sec (\alpha)\max \{\|C\|,\|D\|\} w_q(A).   
		\end{equation}
		Again, replacing $C$ with $CB$ and $D$ with $BD$ in inequality (\ref{inequality}), we obtain that
		\begin{align*}
			w_q(ACB \pm BD A) \leq& \frac{2}{|q|}\sec (\alpha)\max \{\|CB\|,\|BD\|\} w_q(A)\\
			\le&\frac{2}{|q|}\sec (\alpha)\max \{\|C\|,\|D\|\} w_q(A)\|B\|.
		\end{align*}
		This completes the proof.
	\end{proof}
	\begin{remark}
		It was obtained in \cite[Theorem 11]{fong1983unitarily} that
		\begin{equation}\label{newe3}
			w(A B \pm B A) \leq 2 \sqrt{2}w(A)\|B\|.   
		\end{equation}
		Take $C=D=I$ in the aforementioned theorem, we have
		\begin{equation}\label{newe1}
			w_q(A B \pm B A) \leq \frac{2}{|q|} \sec (\alpha) w_q(A)\|B\|.
		\end{equation}
		
		If $q=1$, relation (\ref{newe1}) implies that
		\begin{equation}\label{newe22}
			w(A B \pm B A) \leq 2 \sec (\alpha) w(A)\|B\|.
		\end{equation}
		
		Thus, relation (\ref{newe22}) is a refinement of the relation (\ref{newe3}) if $\alpha \in\left[0, \frac{\pi}{4}\right)$.  
	\end{remark}
	Replacing $A$ and $B$ in relation (\ref{newe1}), we can obtain the following corollary easily.
	\begin{corollary}
		If $A,B \in \prod_{s,\alpha}^n$ and $q \in \mathcal{D'}$. Then  
		\begin{equation*}
			|q|w_q(A B \pm B A) \leq 2 \sec (\alpha) \min \{w_q(A)\|B\|,w_q(B)\|A\|\}.
		\end{equation*} 
	\end{corollary}
	

	It is noteworthy that the literature does not extensively address non-integral powers of the $q$-numerical radius. In our forthcoming results, we address this gap by establishing a relation between $w^t_q(A)$ and $w_q(A^t)$ for $A \in \prod_{s,\alpha}^n$. It is well-known that $\|A^t\|=\|A\|^t$ for positive matrix $A$.
	The significance of the forthcoming result lies in its ability to establish a relationship between $w_q(A^{-1})$ and $w_q^{-1}(A)$. In general there is no such relation between $w_q(A^{-1})$ and $w_q^{-1}(A)$. For example, let $A=\begin{bmatrix}
		1 & 0\\
		0 & 2
	\end{bmatrix}$ and $q \in (0,1)$. Then,  $w_q(A)=\frac{3q+1}{2}$ and $ w_q(A^{-1})=\frac{3q+1}{4}$. Here $w_q(A^{-1}) \le w_q^{-1}(A)$ if $q \in (0, 0.6095]$ and $w_q^{-1}(A) \le w_q(A^{-1})$ if $q\in [0.6095,1)$.
	\begin{theorem}\label{tthm}
		If $ A \in \prod_{s,\alpha}^n$, $q \in \mathcal{D'}$ and $t \in [0,1]$ then
		\begin{itemize}
			\item[(a)] $|q|^{t+1}\cos^t(\alpha)w^t_q(A) \le |q|^tw_q(A^t) \le \sec^{2t}(\alpha)\sec(t \alpha)w^t_q(A),$
			\item[(b)] $|q|^{t+1}\cos(t \alpha)\cos^{2t}(\alpha)w_{q}^{-t}(A)\le w_q(A^{-t}).$
		\end{itemize}   
	\end{theorem}
	\begin{proof}
		\begin{itemize}
			\item [(a)] By Lemma \ref{normr1}, we have
			\begin{equation*}
				\cos(\alpha) w(A) \le \|\mathcal{R}(A)\|.
			\end{equation*}  
			For $t \in [0,1]$, Lemma \ref{il1.2} implies, 
			\begin{equation*}
				w_q(A^t)\le \sec{(t \alpha)}\|\mathcal{R}(A^t)\|.
			\end{equation*}
			Now, 
			\begin{align*}
				w_q(A^t)
				\le& \sec(t\alpha)\sec^{2t}(\alpha) \|\mathcal{R}^t(A)\|\\
				= & \sec(t\alpha)\sec^{2t}(\alpha)\|\mathcal{R}(A)\|^t\\
				\le &  \left( \frac{1}{|q|}\right) ^t\sec(t\alpha)\sec^{2t}(\alpha)w^t_q(\mathcal{R}(A)) \\
				\le &  \left( \frac{1}{|q|}\right) ^t\sec(t\alpha)\sec^{2t}(\alpha)w^t_q(A),
			\end{align*}
			where the first inequality follow from  Lemma \ref{il1.7} respectively and the last two inequalities follow from Corollary \ref{co1.6} and relation (\ref{ri}) respectively.

			For the other inequality, we have
			\begin{align*}
				w_q(A^t) \ge & w_q(\mathcal{R}(A^t)) \\ \ge & {|q|}\|\mathcal{R}(A^t)\|\\
				\ge& {|q|}\|\mathcal{R}^t(A)\|\\
				\ge & |q|\|\mathcal{R}(A)\|^t\\
				\ge & |q|\cos^t(\alpha) \|A\|^t\\
				\ge & |q|\cos^t(\alpha) w^t_q(A)
			\end{align*}
			where the first, second, third, and fifth inequalities follow by applying relation (\ref{ri}), Corollary \ref{co1.6}, Lemma {\ref{il1.7}}, and Lemma \ref{normr1} respectively.
			Finally, we have 
			\begin{equation*}
				|q|\cos^t(\alpha) w^t_q(A) \le w_q(A^t) \le \frac{1}{|q|^t}\sec^{2t}(\alpha)\sec(t \alpha)w^t_q(A).
			\end{equation*}
			\item[(b)]  	Using Theorem \ref{t1.16} and Lemma \ref{normr1}, we obtain that 			
			\begin{align*}
				w_q(A^{-t}) \ge & |q|\cos(t \alpha) \|A^{-t}\| \\
				\ge & |q|\cos(t \alpha) \|\mathcal{R}(A^{-t})\|.
			\end{align*}
			Using Lemma \ref{il1.9}, Corollary \ref{co1.6}(a) and the fact that $\|A\|^{-1} \le \|A^{-1}\|$, where $A$ is invertible matrix, we have 
			\begin{align*}
				w_q(A^{-t})	\ge & |q|\cos(t \alpha) \cos^{2t}(\alpha)\|\mathcal{R}^{-t}(A)\|\\
				\ge & |q|\cos(t \alpha) \cos^{2t}(\alpha)\|\mathcal{R}(A)\|^{-t}\\
				\ge & |q|^{t+1}\cos(t \alpha) \cos^{2t}(\alpha)w^{-t}_q(\mathcal{R}(A))\\
				\ge& |q|^{t+1}\cos(t \alpha) \cos^{2t}(\alpha)w^{-t}_q(A).
			\end{align*}

			Hence the required result.   
		\end{itemize}
		
		
	\end{proof}
	\begin{remark}
		\begin{itemize}
			\item [(i)] For $q=1$, the lower bound in the part $(a)$ of the aforementioned theorem represents an improvement over the lower bound mentioned in \cite[Theorem 3.1 ]{bedrani2021numerical} which is as follows
			\begin{equation*}
				\cos(t\alpha)\cos^t(\alpha) w^t(A) \le w(A^t) \le \sec^{2t}(\alpha)\sec(t \alpha)w^t(A).
			\end{equation*}
			\item[(ii)]   Theorem \ref{t1.15} and Theorem \ref{lt}(a) give us the following significant relation,
			\begin{equation}\label{powert}
				|q|^{t+2} \|A\|^t \le |q|w_q(A^t) \le  {\|A\|}^t,
			\end{equation}
			where $A$ is positive matrix.
			If $q=1$ then relation \eqref{powert} gives us the well-known equality $\|A^t\|=\|A\|^t,$ $A$ is positive matrix.
			\item[(iii)] The case $t=1$ in Theorem \ref{tthm}$(b)$ provides
			\begin{equation}\label{num}
				|q|^2\cos^3(\alpha)w_q^{-1}(A) \le w_q(A^{-1}),
			\end{equation}
			and if $A$ is positive matrix, relation (\ref{num}) implies that
			\begin{equation*}
				|q|^2w^{-1}_q(A)  \le w_q(A^{-1}).
			\end{equation*}
		\end{itemize}  
	\end{remark}



\section{$q$-Numerical Radius Inequalities of $2 \times 2$ Operator Matrices}
Assuming $\mathcal{H}$ represents a complex Hilbert space equipped with the inner product $\langle .,. \rangle$, the direct sum $\mathcal{H} \oplus \mathcal{H}$ constructs another Hilbert space, and any operator $ \mathbf{T}\in \mathcal{B}({\mathcal{H} \oplus \mathcal{H}})$ can be represented by a $2 \times 2$ operator matrix in the following manner
\[\mathbf{T}=\begin{bmatrix}
	{Z} & {X}\\
	{Y} & {W}
\end{bmatrix}
\]
where $Z, X, Y,$ and $W$ are in $\mathcal{B}(\mathcal{H})$. In this section, our objective is to analyze the properties of the $q$- numerical radius of operators having off-diagonal representation of the form $\begin{bmatrix}
	{0} & {X}\\
	{Y} & {0}
\end{bmatrix}
$.
Since the $q$-numerical radius is weakly unitarily invariant, i.e. \begin{equation*}
	W_q(U^*TU) = W_q(T)
\end{equation*} for any unitary operator $U$ on $\mathcal{H}$, the following relations can be easily deducted by considering the unitary operators $\begin{bmatrix}
	{I} & {0}\\
	{0} & {e^{-\frac{i\theta}{2}}I}
\end{bmatrix}$, 
$\theta \in \mathbb{R}$ and  $\begin{bmatrix}
	{0} & {I}\\
	{I} & {0}
\end{bmatrix}$ respectively.
\begin{equation}\label{t5.1a}
	w_q \left(\begin{bmatrix}
		{0} & {X}\\
		{Y} & {0}
	\end{bmatrix}\right)=w_q \left(\begin{bmatrix}
		{0} & {X}\\
		{e^{i \theta}Y} & {0}\\
	\end{bmatrix}
	\right)
	~\text{for all}~ \theta \in \mathbb{R}
\end{equation}   
and 
\begin{equation}\label{t5.1b}
	w_q \left(\begin{bmatrix}
		{0} & {X}\\
		{Y} & {0}
	\end{bmatrix}\right)=w_q \left(\begin{bmatrix}
		{0} & {Y}\\
		{X} & {0}\\
	\end{bmatrix}
	\right)
\end{equation}
			Next, some observations of the $q$-numerical radius of $\begin{bmatrix}
				0 & A\\
				A & 0
			\end{bmatrix}$ are mentioned, where $A$ is a hermitian matrix.
			Let $\lambda_1\ge\lambda_2 \ge...\ge \lambda_n$ are the eigenvalues of $A$. Thus by using Theorem 3.4 \cite[p.384]{gau2021numerical}, we have
			$$w_q(A)=\frac{|q|}{2}|\lambda_1+\lambda_n|+\frac{1}{2}|\lambda_1-\lambda_n|.$$
			Therefore, the largest eigenvalue of $\begin{bmatrix}
				0 & A \\
				A & 0
			\end{bmatrix}$ is $\lambda_{max}=\max\{-\lambda_n,\lambda_1\}$ and the smallest eigenvalue of $\begin{bmatrix}
				0 & A \\
				A & 0
			\end{bmatrix}$ is $\lambda_{min}=\min\{-\lambda_1, \lambda_n\}$. Therefore,
			$$w_q\left(\begin{bmatrix}
				0 & A \\
				A & 0
			\end{bmatrix}\right)=\left\lbrace \begin{aligned}
				\frac{|q|}{2}|\lambda_1+\lambda_n|+\frac{1}{2}|\lambda_1-\lambda_n|~\hspace{0.5cm}&\text{if}~\lambda_{max}=\lambda_1,\lambda_{min}=\lambda_n\\
				\frac{|q|}{2}|\lambda_1+\lambda_n|+\frac{1}{2}|\lambda_1-\lambda_n|~\hspace{0.5cm}&\text{if}~\lambda_{max}=-\lambda_n,\lambda_{min}=-\lambda_1\\
				| \lambda_1| ~\hspace{0.5cm}&\text{if}~\lambda_{max}=\lambda_1,\lambda_{min}=-\lambda_1\\
				| \lambda_n| ~\hspace{0.5cm}&\text{if}~\lambda_{max}=-\lambda_n,\lambda_{min}=\lambda_n.
			\end{aligned}\right.$$
			and $w_q(A) \le w_q\left(\begin{bmatrix}
				0 & A \\
				A & 0
			\end{bmatrix}\right)$. 
			
			In particular, let $A$ be a positive $n \times n$ matrix with eigenvalues $\lambda_1\ge\lambda_2 \ge...\ge \lambda_n \ge 0$, we have $\lambda_{max}=\lambda_1$, $\lambda_{min}=-\lambda_1$ and
			\begin{equation*}
				w_q \left( \begin{bmatrix}
					0 & A \\
					A & 0
				\end{bmatrix}\right)=\|A\|=w(A).
			\end{equation*}

			Following this, we will establish the bounds for $q$-numerical radius of $\begin{bmatrix}
				{0} & {Y}\\
				{Y} & {0}
			\end{bmatrix}$ and $\begin{bmatrix}
				{0} & {X}\\
				{Y} & {0}
			\end{bmatrix}$, $X,Y \in \mathcal{B(H)}$. For $Y \in \mathcal{B(H)}$, $\begin{bmatrix}
				{0} & {Y}\\
				{Y} & {0}
			\end{bmatrix}$ satisfies the equality 
			$ w\left(\begin{bmatrix}
				{0} & {Y}\\
				{Y} & {0}
			\end{bmatrix}\right)=w(Y)$ \cite{hirzallah2011numerical}.
			However, a similar assertion may not hold for $q$-numerical radius, as illustrated in the subsequent example.

			Let
			$$\mathbf{T}=\begin{bmatrix}
				0&0&2&0\\
				0&0&0&3\\
				2&0&0&0\\
				0 & 3&0&0
			\end{bmatrix}.$$
			Then $w_q(\mathbf{T})=3$ and $w_q\left( \begin{bmatrix}
				2 & 0\\
				0 &3
			\end{bmatrix}\right)=\frac{5|q|}{2}+\frac{1}{2}$. Clearly,
			$w_q(\mathbf{T})\ne w_q\begin{bmatrix}
				2 & 0\\
				0 &3
			\end{bmatrix}$ except $|q|=1$ but $ w_q\begin{bmatrix}
				2 & 0\\
				0 &3
			\end{bmatrix} \le w_q(\mathbf{T}).$
			
			With this observation, we present the following result.
			
			\begin{proposition}\label{l5.3}
				Let $X,Y \in \mathcal{B(H)}$ and $q \in \mathcal{D'}$, we have 
				\begin{itemize}
					\item [(a)]
					$\max\{w_q(X-Y),w_q(X+Y)\} \le w_q\left(\begin{bmatrix}
						{X} & {Y}\\
						{Y} & {X}
					\end{bmatrix}\right) \le \{ \|X-Y\|, \|X+Y\|\},
					$
					\item [(b)]
					$
					\frac{1}{2}\max\{w_q(X+Y),w_q(X-Y)\}
					\le w_q \left(
					\begin{bmatrix}
						{0} & {X}\\
						{Y} & {0}
					\end{bmatrix}\right)
					\le \frac{1}{2}(\|X+Y\|+\|X-Y\|).$
					
				\end{itemize}		
			\end{proposition}
			\begin{proof}
				\begin{itemize}
					\item [(a)] From \cite[Remark 4(i)]{patra2024estimation}, we have
					$$
					\max\{w_q(X-Y),w_q(X+Y)\} \le w_q\left(\begin{bmatrix}
						{X-Y} & {0}\\
						{0} & {X+Y}
					\end{bmatrix}\right) \le \{ \|X-Y\|, \|X+Y\|\}.
					$$
					As $w_q(U^*TU)=w_q(T)$, taking $U=\frac{1}{\sqrt{2}}\begin{bmatrix}
						{I} & {I}\\
						{-I} & {I}
					\end{bmatrix},$	
					we have $\begin{bmatrix}
						{X} & {Y}\\
						{Y} & {X}
					\end{bmatrix} $ is unitarily similar to $\begin{bmatrix}
						{X-Y} & {0}\\
						{0} & {X+Y}
					\end{bmatrix}$. This implies,
					\[
					\max\{w_q(X-Y),w_q(X+Y)\} \le w_q\left(\begin{bmatrix}
						{X} & {Y}\\
						{Y} & {X}
					\end{bmatrix}\right) \le \{ \|X-Y\|, \|X+Y\|\}.
					\]
					\item [(b)]
					The triangle inequality of $q$-numerical radius provides
					\begin{equation*}
						w_q\left(\begin{bmatrix}
							{0} & {X+Y}\\
							{X+Y} & {0}
						\end{bmatrix}\right) \le w_q\left(\begin{bmatrix}
							{0} & {X}\\
							{Y} & {0}
						\end{bmatrix}\right)+w_q\left(\begin{bmatrix}
							{0} & {Y}\\
							{X} & {0}
						\end{bmatrix}\right).
					\end{equation*}
					From the relation $w_q \left(\begin{bmatrix}
						{0} & {X}\\
						{Y} & {0}
					\end{bmatrix}\right)=w_q \left(\begin{bmatrix}
						{0} & {Y}\\
						{X} & {0}\\
					\end{bmatrix}
					\right)$ (equality \ref{t5.1b}), it follows 
							\begin{equation*}
								w_q\left(\begin{bmatrix}
									{0} & {X+Y}\\
									{X+Y} & {0}
								\end{bmatrix}\right) \le 2w_q\left(\begin{bmatrix}
									{0} & {X}\\
									{Y} & {0}
								\end{bmatrix}\right).
							\end{equation*}
							Taking $X=0$ in part (a), we have $w_q(Y) \le w_q\left(\begin{bmatrix}
								{0} & {Y}\\
								{Y} & {0}
							\end{bmatrix}\right)$. Therefore,       
							\begin{equation*}
								w_q(X+Y) \le w_q\left(\begin{bmatrix}
									{0} & {X+Y}\\
									{X+Y} & {0}
								\end{bmatrix}\right) \le 2w_q\left(\begin{bmatrix}
									{0} & {X}\\
									{Y} & {0}
								\end{bmatrix}\right)
								=2w_q\left(\begin{bmatrix}
									{0} & {X}\\
									{e^{i\theta} Y} & {0}
								\end{bmatrix}\right)  .
							\end{equation*}
							Replacing $Y$ with $-Y$, and taking $\theta=\pi$, 
							we have 
							
							\begin{equation*}
								w_q(X-Y) \le w_q\left(\begin{bmatrix}
									{0} & {X-Y}\\
									{X-Y} & {0}
								\end{bmatrix}\right) \le 2w_q\left(\begin{bmatrix}
									{0} & {X}\\
									{Y} & {0}
								\end{bmatrix}\right).
							\end{equation*}
							Thus, 
							\begin{equation*}
								w_q\left(\begin{bmatrix}
									{0} & {X}\\
									{Y} & {0}
								\end{bmatrix}\right) \ge \frac{1}{2}\max\{w_q(X+Y), w_q(X-Y)\}.
							\end{equation*}
							Let $U=\frac{1}{\sqrt{2}}\begin{bmatrix}
								{I} & {-I}\\
								{I} & {I}
							\end{bmatrix}$. Then
							\begin{align*}
								w_q\left(\begin{bmatrix}
									{0} & {X}\\
									{Y} & {0}
								\end{bmatrix}\right) =&w_q\left(U^*\begin{bmatrix}
									{0} & {X}\\
									{Y} & {0}
								\end{bmatrix}U\right)	\\
								=& \frac{1}{2}w_q\left(\begin{bmatrix}
									{X+Y} & {X-Y}\\
									{-(X-Y)} & {-(X+Y)}
								\end{bmatrix}\right)\\
								\le& \frac{1}{2}w_q\left(\begin{bmatrix}
									{X+Y} & {0}\\
									{0} & {-(X+Y)}
								\end{bmatrix} +\begin{bmatrix}
									{0} & {X-Y}\\
									{-(X-Y)} & {0}
								\end{bmatrix}\right)\\
								\le& \frac{1}{2}\left(\norm{\begin{bmatrix}
										{X+Y} & {0}\\
										{0} & {-(X+Y)}
								\end{bmatrix}} +\norm{\begin{bmatrix}
										{0} & {X-Y}\\
										{-(X-Y)} & {0}
								\end{bmatrix}}\right)\\
								=& \frac{1}{2}\left(\|X+Y\|+\|X-Y\| \right).
							\end{align*}
							Hence the required result holds.
							
						\end{itemize}
					\end{proof}

					\begin{remark}
						If we take $X=0$ in part (a) of the aforementioned result, we have
						\begin{equation}\label{sharp}
							w_q(Y) \le w_q\left(\begin{bmatrix}
								{0} & {Y}\\
								{Y} & {0}
							\end{bmatrix}\right)\le \|Y\|.
						\end{equation}
						The bounds mentioned in relation \eqref{sharp}. For $Y=\begin{bmatrix}
							0 &1 \\
							1 & 0
						\end{bmatrix}$, we have $\|Y\|=w_q(Y)=w_q\left( \begin{bmatrix}
							0 &Y \\
							Y & 0
						\end{bmatrix}\right) =1.$ 
					\end{remark}
					\begin{corollary}\label{NEW}
						Let $X \in \mathcal{B(H)}$ and $q \in \mathcal{D'}$. Then
						\begin{itemize}
							\item [(a)]
							$
							|q| \max \{ \|\mathcal{R}(X)\|, \|\mathcal{I}(X)\|\} \le w_q\left(\begin{bmatrix}
								{0} & {X}\\
								{X^*} & {0}
							\end{bmatrix}\right) \le \|\mathcal{R}(X)\|+\|\mathcal{I}(X)\|,
							$
							\item [(b)] 
							$w_q(X)
							\le w_q \left(
							\begin{bmatrix}
								{0} & {\mathcal{R}(X)}\\
								{\mathcal{I}(X)} & {0}
							\end{bmatrix}\right)
							\le \|X\|,$
							\item [(c)] if $X$ is hermitian, then
							$|q| \|X\| \le w_q\left(\begin{bmatrix}
								{0} & {X}\\
								{X} & {0}
							\end{bmatrix}\right) \le \|X\|,$
							\item [(d)] if $X^2=0$, then
							$|q| \|\mathcal{R}(X)\| \le w_q\left(\begin{bmatrix}
								{0} & {X}\\
								{X^*} & {0}
							\end{bmatrix}\right) \le 2\|\mathcal{R}(X)\|.$
						\end{itemize}
					\end{corollary}
					\begin{proof}
						Taking $Y=X^*$ in Proposition \ref{l5.3}(b) and using Corollary \ref{co1.6}(a), we can easily obtain the result mentioned in part (a).  Replacing $Y$ with $iY$ in Proposition \ref{l5.3}(b), we have
						\begin{equation*}
							\frac{1}{2}\max\{w_q(X+iY)+w_q(X-iY)\}
							\le w_q \left(
							\begin{bmatrix}
								{0} & {X}\\
								{iY} & {0}
							\end{bmatrix}\right)
							\le \frac{1}{2}(\|X+iY\|+\|X-iY\|).
						\end{equation*}	
						For $\theta=\frac{\pi}{2}$, from the relation $w_q \left(
						\begin{bmatrix}
							{0} & {X}\\
							{Y} & {0}
						\end{bmatrix}\right)=w_q \left(
						\begin{bmatrix}
							{0} & {X}\\
							{e^{i \theta}Y} & {0}
						\end{bmatrix}\right)$(equality \eqref{t5.1a}), it follows
						\begin{equation}\label{ne3.32}
							\frac{1}{2}\max\{w_q(X+iY)+w_q(X-iY)\}
							\le w_q \left(
							\begin{bmatrix}
								{0} & {X}\\
								{Y} & {0}
							\end{bmatrix}\right)
							\le \frac{1}{2}(\|X+iY\|+\|X-iY\|).
						\end{equation}	
						The inequalities in part (b) follow by replacing $X$ with $\mathcal{R}(X)$ and $Y$ with $\mathcal{I}(X)$ in (\ref{ne3.32}). Moreover, the relations in part (c) and part (d) follow from part (a) when $X$ is hermitian and $X^2=0$($\|\mathcal{R}(X)\|=\|\mathcal{I}(X)\|$) respectively.
					\end{proof}

					Let $T$ be a bounded linear operator, acting on a Hilbert space $\mathcal{H}$, there exists a unique complex number $c\in\overline{W(T)}$ such that \cite{stampfli1970norm}
					\begin{equation}\label{em}
						m(T)=\inf_{\lambda \in \mathbb{C}}\|T-\lambda I\|=\|T-cI\|.
					\end{equation}
					Prasanna \cite{prasanna1981norm} termed $m(T)$ as the transcendental radius and defined it as 
					\begin{equation}\label{em2}
						m^2(T)=\sup_{\|x\|=1}\left(\|Tx\|^2-|\langle Tx,x \rangle|^2\right).
					\end{equation}
					Let us consider an orthonormal set $\mathcal{O}$ containing $x$ and $z$. Bessel's inequality implies 
					\begin{eqnarray*}
						\sum_{y' \in \mathcal{O}}|\langle Tx,y' \rangle |^2 \le \|Tx\|^2.\\
						|\langle Tx,z \rangle |^2 \le \sum_{y' \in \mathcal{O} \setminus \{x\}}|\langle Tx,y' \rangle |^2 \le \|Tx\|^2-|\langle Tx,x \rangle|^2 . 
					\end{eqnarray*}
					Therefore, relations \eqref{em} and \eqref{em2} imply,
					\begin{equation}\label{bessel}
						| \langle Tx,z \rangle | \le (\|Tx\|^2-|\langle Tx,x \rangle|^2)^{\frac{1}{2}}
						\le m(T) 
						\le \|T-\lambda I \|, ~\lambda \in \mathbb{C}.   
					\end{equation}
					Using $m(T)$, the following result gives us an upper bound of $w_q \left(\begin{bmatrix}
						{0} & {X}\\
						{Y} & {0}
					\end{bmatrix}\right)$ using the concept of non-negative functions $f$ and $g$. The following lemma is used in the next theorem.
					\begin{lemma}  \cite[Theorem 1]{kittaneh1988notes}\label{l5.8}
						Let $T \in \mathcal{B(H)}$, $f$ and $g$ be non-negative functions on $\mathclose{[}0, \infty\mathopen{)}$ which are continuous and satisfying the relation $f(t)g(t)=t$ for all $t \in\mathclose{[}0, \infty\mathopen{)}$. Then
						$
						|\langle Tx,y \rangle|\le \|f(|T|)x\| \|g(|T^*|)y\|
						$
						for all $x,y \in \mathcal{H}$.
					\end{lemma}
					\begin{theorem}\label{non-negative}
						Let $f$ and $g$ satisfies the conditions of Lemma \ref{l5.8}, $q \in \mathcal{D'}$ and $\lambda$ is any arbitrary complex number. Then, we have
						\begin{itemize}
							\item [(a)] \begin{align*}
								w_q \left(\begin{bmatrix}
									{0} & {X}\\
									{Y} & {0}
								\end{bmatrix}\right)
								\le & \frac{1}{2} \max \{\|f^2(|Y|)+|q|^2g^2(|X^*|)\|,\|f^2(|X|)+|q|^2g^2(|Y^*|) \|\}\\
								+&\frac{(1-|q|^2)}{2} \max\{ \|g(|X^*|)\|^2, \|g(|Y^*|)\|^2\}\\
								+&|q|\sqrt{1-|q|^2} \max \{ \|\{g^2|X^*|-\lambda I\|,\|g^2|Y^*|-\lambda I\|\}.
							\end{align*}
							\item [(b)]
							\begin{align*}
								w_q^2 \left(\begin{bmatrix}
									{0} & {X}\\
									{Y} & {0}
								\end{bmatrix}\right)
								\le & \frac{1}{2} \max \{\|f^4(|Y|)+|q|^2g^4(|X^*|)\|,\|f^4(|X|)+|q|^2g^4(|Y^*|) \|\}\\
								+&\frac{(1-|q|^2)}{2} \max\{ \|g(|X^*|)\|^4, \|g(|Y^*|)\|^4\}\\
								+&{|q|\sqrt{1-|q|^2}} \max \{\|g^4|X^*|-\lambda I\|,\|g^4|Y^*|-\lambda I\|\}.
							\end{align*}
						\end{itemize}
					\end{theorem}
					\begin{proof}
						\begin{itemize}
							\item[(a)] 
							Let
							Let $x=\begin{bmatrix}
								{x_1}\\
								{x_2} 
							\end{bmatrix} \in \mathcal{H},$
							and $y=
							\begin{bmatrix}
								{y_1}\\
								{y_2} 
							\end{bmatrix} \in \mathcal{H}$ with $\|x\|=\|y\|=1$ and $\langle x,y \rangle=q$.
							From Lemma \ref{l5.8}, we have
							\begin{align*}
								\left| \left \langle \begin{bmatrix}
									{0} & {X}\\
									{Y} & {0}
								\end{bmatrix}x,y \right \rangle \right| \le& \norm{f\left(  \left|\begin{bmatrix}
										{0} & {X}\\
										{Y} & {0}
									\end{bmatrix} \right| \right)x}	\norm{g
									\left(  \left|\begin{bmatrix}
										{0} & {Y^*}\\
										{X^*} & {0}
									\end{bmatrix} \right| \right)y}	\\
								\le & \left \langle f^2
								\left|\begin{bmatrix}
									{0} & {X}\\
									{Y} & {0}
								\end{bmatrix} \right| x,x \right \rangle^\frac{1}{2}\left \langle g^2 \left|\begin{bmatrix}
									{0} & {Y^*}\\
									{X^*} & {0}
								\end{bmatrix} \right| y,y \right \rangle^\frac{1}{2}\\
								\le & \left \langle f^2\begin{bmatrix}
									{|Y|} & {0}\\
									{0} & {|X|}
								\end{bmatrix} x,x \right \rangle^\frac{1}{2}\left \langle g^2\begin{bmatrix}
									{|X^*|} & {0}\\
									{0} & {|Y^*|}
								\end{bmatrix} y,y \right \rangle^\frac{1}{2}\\
								\le & \frac{1}{2}\left(\left \langle f^2\begin{bmatrix}
									{|Y|} & {0}\\
									{0} & {|X|}
								\end{bmatrix} x,x \right \rangle+\left \langle g^2\begin{bmatrix}
									{|X^*|} & {0}\\
									{0} & {|Y^*|}
								\end{bmatrix} y,y \right \rangle\right).
							\end{align*} 
							
							Putting $y=\overline{q}x+\sqrt{1-|q|^2}z$, where $\|z\|=1$ and $\langle x,z \rangle =0$. We obtain
							\begin{align*}
								&\left| \left \langle \begin{bmatrix}
									{0} & {X}\\
									{Y} & {0}
								\end{bmatrix}x,y \right \rangle \right|\\
								\le& \frac{1}{2}\left(\left \langle f^2\begin{bmatrix}
									{|Y|} & {0}\\
									{0} & {|X|}
								\end{bmatrix} x,x \right \rangle+\left \langle g^2\begin{bmatrix}
									{|X^*|} & {0}\\
									{0} & {|Y^*|}
								\end{bmatrix}\overline{q}x+\sqrt{1-|q|^2}z,\overline{q}x+\sqrt{1-|q|^2}z \right \rangle\right)\\
								\le& \frac{1}{2}\left \langle f^2\begin{bmatrix}
									{|Y|} & {0}\\
									{0} & {|X|}
								\end{bmatrix} x,x \right \rangle + \frac{|q|^2}{2}\left \langle g^2\begin{bmatrix}
									{|X^*|} & {0}\\
									{0} & {|Y^*|}
								\end{bmatrix}x,x \right \rangle
								+ \frac{1-|q|^2}{2}\left \langle g^2\begin{bmatrix}
									{|X^*|} & {0}\\
									{0} & {|Y^*|}
								\end{bmatrix}z,z \right \rangle\\
								+&\frac{1}{2}\left( \overline{q}\sqrt{1-|q|^2}\left \langle g^2\begin{bmatrix}
									{|X^*|} & {0}\\
									{0} & {|Y^*|}
								\end{bmatrix}x,z \right \rangle + q\sqrt{1-|q|^2}\left \langle g^2\begin{bmatrix}
									{|X^*|} & {0}\\
									{0} & {|Y^*|}
								\end{bmatrix}z,x \right \rangle\right)\\
									\le & \frac{1}{2}\left \langle f^2\begin{bmatrix}
										{|Y|} & {0}\\
										{0} & {|X|}
									\end{bmatrix} x,x \right \rangle + \frac{|q|^2}{2}\left \langle g^2\begin{bmatrix}
										{|X^*|} & {0}\\
										{0} & {|Y^*|}
									\end{bmatrix}x,x \right \rangle\\
									+& \frac{1-|q|^2}{2}\left \langle g^2\begin{bmatrix}
										{|X^*|} & {0}\\
										{0} & {|Y^*|}
									\end{bmatrix}z,z \right \rangle
									+ \mathcal{R}\left(\overline{q}\sqrt{1-|q|^2} \left \langle g^2\begin{bmatrix}
										{|X^*|} & {0}\\
										{0} & {|Y^*|}
									\end{bmatrix}x,z \right \rangle \right)\\
									\le & \frac{1}{2}\left \langle f^2\begin{bmatrix}
										{|Y|} & {0}\\
										{0} & {|X|}
									\end{bmatrix} x,x \right \rangle + \frac{|q|^2}{2}\left \langle g^2\begin{bmatrix}
										{|X^*|} & {0}\\
										{0} & {|Y^*|}
									\end{bmatrix}x,x \right \rangle\\
									+& \frac{1-|q|^2}{2}\left \langle g^2\begin{bmatrix}
										{|X^*|} & {0}\\
										{0} & {|Y^*|}
									\end{bmatrix}z,z \right \rangle
									+{|q|\sqrt{1-|q|^2}}\left|\left \langle g^2\begin{bmatrix}
										{|X^*|} & {0}\\
										{0} & {|Y^*|}
									\end{bmatrix}x,z \right \rangle \right|.
								\end{align*} 
								Using relation \eqref{bessel}, we have
								\begin{align*}
									\left| \left \langle \begin{bmatrix}
										{0} & {X}\\
										{Y} & {0}
									\end{bmatrix}x,y \right \rangle \right|
									\le & \frac{1}{2}\left \langle f^2\begin{bmatrix}
										{|Y|} & {0}\\
										{0} & {|X|}
									\end{bmatrix} x,x \right \rangle + \frac{|q|^2}{2}\left \langle g^2\begin{bmatrix}
										{|X^*|} & {0}\\
										{0} & {|Y^*|}
									\end{bmatrix}x,x \right \rangle\\
									+& \frac{1-|q|^2}{2}\left \langle g^2\begin{bmatrix}
										{|X^*|} & {0}\\
										{0} & {|Y^*|}
									\end{bmatrix}z,z \right \rangle\\
									+&{|q|\sqrt{1-|q|^2}} \norm{g^2\begin{bmatrix}
											{|X^*|} & {0}\\
											{0} & {|Y^*|}
										\end{bmatrix} -\lambda \begin{bmatrix}
											{I} & {0}\\
											{0} & {I}
									\end{bmatrix}}, ~ \text{where}~ \lambda \in \mathbb{C}\\
									\le & \frac{1}{2}\left \langle \begin{bmatrix}
										{f^2(|Y|)} & {0}\\
										{0} & {f^2(|X|)}
									\end{bmatrix} x,x \right \rangle + \frac{|q|^2}{2}\left \langle \begin{bmatrix}
										{g^2(|X^*|)} & {0}\\
										{0} & {g^2(|Y^*|)}
									\end{bmatrix}x,x \right \rangle\\
									+& \frac{1-|q|^2}{2}\left \langle \begin{bmatrix}
										{g^2(|X^*|)} & {0}\\
										{0} & {g^2(|Y^*|)}
									\end{bmatrix}z,z \right \rangle\\
									+&{|q|\sqrt{1-|q|^2}} \norm{\begin{bmatrix}
											{g^2(|X^*|)} & {0}\\
											{0} & {g^2(|Y^*|)}
										\end{bmatrix} -\lambda \begin{bmatrix}
											{I} & {0}\\
											{0} & {I}
									\end{bmatrix}}\\
									\le & \frac{1}{2}\left \langle \begin{bmatrix}
										{f^2(|Y|)+|q|^2g^2(|X^*|)} & {0}\\
										{0} & {f^2(|X|)+|q|^2g^2(|Y^*|)}
									\end{bmatrix} x,x \right \rangle \\
									+& \frac{1-|q|^2}{2}\left \langle \begin{bmatrix}
										{g^2(|X^*|)} & {0}\\
										{0} & {g^2(|Y^*|)}
									\end{bmatrix}z,z \right \rangle\\
									+&{|q|\sqrt{1-|q|^2}} \norm{\begin{bmatrix}
											{g^2(|X^*|)} & {0}\\
											{0} & {g^2(|Y^*|)}
										\end{bmatrix} -\lambda \begin{bmatrix}
											{I} & {0}\\
											{0} & {I}
									\end{bmatrix}}
								\end{align*}
								Therefore,
								\begin{align*}
									\left| \left \langle \begin{bmatrix}
										{0} & {X}\\
										{Y} & {0}
									\end{bmatrix}x,y \right \rangle \right| 
									\le & \frac{1}{2}\left(\norm{ \begin{bmatrix}
											f^2(|Y|)+|q|^2g^2(|X^*|) & {0}\\
											{0} & {f^2(|X|)+|q|^2g^2(|Y^*|)}
									\end{bmatrix} } \right) \\
									+& \frac{1-|q|^2}{2}\norm{ \begin{bmatrix}
											g^2(|X^*|) & {0}\\
											{0} & {g^2(|Y^*|)}
									\end{bmatrix}}\\
									+&{|q|\sqrt{1-|q|^2}} \norm{\begin{bmatrix}
											{g^2|X^*|} & {0}\\
											{0} & {g^2|Y^*|}
										\end{bmatrix} -\lambda\begin{bmatrix}
											{I} & {0}\\
											{0} & {I}
									\end{bmatrix}}  \\ 
									= & \frac{1}{2} \max \{\|f^2(|Y|)+|q|^2g^2(|X^*|)\|,\|f^2(|X|)+|q|^2g^2(|Y^*|) \|\}\\
									+&\frac{1-|q|^2}{2} \max\{ \|g^2(|X^*|)\|, \|g^2(|Y^*|)\|\}\\
									+&{|q|\sqrt{1-|q|^2}} \max \{\|g^2|X^*|-\lambda I\|,\|g^2|Y^*|-\lambda I\|\}.
								\end{align*}
								Taking supremum for all $x,y \in \mathcal{H}$ with $\|x\|=\|y\|=1$ and $\langle x,y \rangle=q$, we obtain 
								\begin{align*}
									w_q \left(\begin{bmatrix}
										{0} & {X}\\
										{Y} & {0}
									\end{bmatrix}\right)
									\le & \frac{1}{2} \max \{\|f^2(|Y|)+|q|^2g^2(|X^*|)\|,\|f^2(|X|)+|q|^2g^2(|Y^*|) \|\}\\
									+&\frac{1-|q|^2}{2} \max\{ \|g^2(|X^*|)\|, \|g^2(|Y^*|)\|\}\\
									+&{|q|\sqrt{1-|q|^2}}\max \| \{g^2|X^*|-\lambda I\|,\|g^2|Y^*|-\lambda I\|\}.
								\end{align*}
								\item [(b)]
								From Lemma \ref{l5.8}, it follows 
								\begin{align*}
									\left| \left \langle \begin{bmatrix}
										{0} & {X}\\
										{Y} & {0}
									\end{bmatrix}x,y \right \rangle \right|^2 
									\le& \norm{f\left(  \left|\begin{bmatrix}
											{0} & {X}\\
											{Y} & {0}
										\end{bmatrix} \right| \right)x}^2	\norm{g
										\left(  \left|\begin{bmatrix}
											{0} & {Y^*}\\
											{X^*} & {0}
										\end{bmatrix} \right| \right)y}^2	\\
									\le & \left \langle f^2\begin{bmatrix}
										{|Y|} & {0}\\
										{0} & {|X|}
									\end{bmatrix} x,x \right \rangle\left \langle g^2\begin{bmatrix}
										{|X^*|} & {0}\\
										{0} & {|Y^*|}
									\end{bmatrix} y,y \right \rangle\\
									\le & \frac{1}{2}\left(\left \langle f^2\begin{bmatrix}
										{|Y|} & {0}\\
										{0} & {|X|}
									\end{bmatrix} x,x \right \rangle^2+\left \langle g^2\begin{bmatrix}
										{|X^*|} & {0}\\
										{0} & {|Y^*|}
									\end{bmatrix} y,y \right \rangle^2\right)\\
									\le & \frac{1}{2}\left(\left \langle f^4\begin{bmatrix}
										{|Y|} & {0}\\
										{0} & {|X|}
									\end{bmatrix} x,x \right \rangle+\left \langle g^4\begin{bmatrix}
										{|X^*|} & {0}\\
										{0} & {|Y^*|}
									\end{bmatrix} y,y \right \rangle\right)
								\end{align*}
								A similar calculation as part $(a)$ follows the required result.
								\end{itemize}
							\end{proof}
							\begin{remark}
								Several upper bounds of $w_q \left(\begin{bmatrix}
									{0} & {X}\\
									{Y} & {0}
								\end{bmatrix}\right)$ follow by choosing particular $f$ and $g$.
								If $f(t)=t^{\gamma}$ and $g(t)=t^{1-\gamma}$, $0 \le \gamma \le 1$. Then Theorem \ref{non-negative}(a) gives us, 
								\begin{align*}
									w_q \left(\begin{bmatrix}
										{0} & {X}\\
										{Y} & {0}
									\end{bmatrix}\right)
									\le & \frac{1}{2} \max \{\||Y|^{2\gamma}+|q|^2|X^*|^{2(1-\gamma)}\|,\||X|^{2\gamma}+|q|^2|Y^*|^{2(1-\gamma)} \|\}\\
									+&\frac{1-|q|^2}{2} \max\{ \||X^*|^{2(1-\gamma)}\|, \||Y^*|^{2(1-\gamma)}\|\}\\
									+&{|q|\sqrt{1-|q|^2}} \max \{\||X^*|^{2(1-\gamma)}-\lambda I\|,\||Y^*|^{2(1-\gamma)}-\lambda I\|\}.
								\end{align*}
								If we take $\gamma=\frac{1}{2}$ in the aforementioned inequality, then 
								\begin{align*}
									w_q \left(\begin{bmatrix}
										{0} & {X}\\
										{Y} & {0}
									\end{bmatrix}\right)
									\le & \frac{1}{2} \max \{\||Y|+|q|^2|X^*|\|,\||X|+|q|^2|Y^*| \|\}\\
									+&\frac{1-|q|^2}{2} \max\{ \||X^*|\|, \||Y^*|\|\}\\
									+&{|q|\sqrt{1-|q|^2}} \max \{\||X^*|-\lambda I\|,\||Y^*|-\lambda I\|\}.
								\end{align*}
								When $q=1$, we have
								\begin{equation*}
									w \left(\begin{bmatrix}
										{0} & {X}\\
										{Y} & {0}
									\end{bmatrix}\right)
									\le  \frac{1}{2} \max \{\|f^2(|Y|)+|q|^2g^2(|X^*|)\|,\|f^2(|X|)+|q|^2g^2(|Y^*|) \|\}
								\end{equation*}
								which is mentioned in Theorem 1
								If we take $f(t)=\frac{t}{1+t}$ and $g(t)=1+t$ in Theorem \ref{non-negative}(a), we have
								\begin{align*}
									w_q \left(\begin{bmatrix}
										{0} & {X}\\
										{Y} & {0}
									\end{bmatrix}\right)
									\le & \frac{1}{2} \max \left\lbrace \norm{\left( \frac{|Y|}{I+|Y|}\right)^2+|q|^2\left(I+|X^*|\right)^2}, \norm{\left( \frac{|X|}{I+|X|}\right)^2+|q|^2\left(I+|Y^*|\right)^2} \right\rbrace\\
									+&{|q|\sqrt{1-|q|^2}}\max \| \{(I+|X^*|)^2-\lambda I\|,\|(I+|Y^*|)^2-\lambda I\|\}\\
									+&\frac{(1-|q|^2)}{2} \max\{ \|(I+|X^*|)^2 \|, \|(I+|Y^*|)^2\|\}.
								\end{align*}
							\end{remark}
							The following result provides a lower bound of the $q$-numerical radius of $\begin{bmatrix}
								{0} & {X}\\
								{Y} & {0}
							\end{bmatrix}$.
							
							\begin{theorem}\label{thmnew}
								Let $X,Y \in \mathcal{B(H)}$ and $q \in \mathcal{D'}$, we have
								\begin{equation*}
									w_q \left(\begin{bmatrix}
										{0} & {X}\\
										{Y} & {0}
									\end{bmatrix}\right) 
									\ge \frac{|q|}{2}\max\{\|X\|, \|Y\| \}+ \frac{q}{4}|\|X+Y^*\|-\|X-Y^*\||.
								\end{equation*}
							\end{theorem}
							\begin{proof}
								Corollary \ref{co1.6} implies that
								$w_q \left(\begin{bmatrix}
									{0} & {X}\\
									{Y} & {0}
								\end{bmatrix}\right) \ge |q| \norm{\mathcal{R}\left(\begin{bmatrix}
										{0} & {X}\\
										{Y} & {0}
									\end{bmatrix}\right)}$ and $w_q \left(\begin{bmatrix}
									{0} & {X}\\
									{Y} & {0}
								\end{bmatrix}\right) \ge |q| \norm{\mathcal{I}\left(\begin{bmatrix}
										{0} & {X}\\
										{Y} & {0}
									\end{bmatrix}\right)}$.
								This implies,
								\begin{align*}
									w_q \left(\begin{bmatrix}
										{0} & {X}\\
										{Y} & {0}
									\end{bmatrix}\right) \ge |q| \norm{\mathcal{R}\left(\begin{bmatrix}
											{0} & {X}\\
											{Y} & {0}
										\end{bmatrix}\right)}= & \frac{|q|}{2}\norm{\left(\begin{bmatrix}
											{0} & {X+Y^*}\\
											{Y+X^*} & {0}
										\end{bmatrix}\right)}\\
									=& \frac{|q|}{2}\max \{ \|X+Y^*\|,\|Y+X^* \|\}\\
									=& \frac{|q|}{2}\max  \|X+Y^*\|.
								\end{align*}	
								Also,
								\begin{align*}
									w_q \left(\begin{bmatrix}
										{0} & {X}\\
										{Y} & {0}
									\end{bmatrix}\right) \ge |q| \norm{\mathcal{I}\left(\begin{bmatrix}
											{0} & {X}\\
											{Y} & {0}
										\end{bmatrix}\right)}= & \frac{|q|}{2}\norm{\left(\begin{bmatrix}
											{0} & {X-Y^*}\\
											{Y-X^*} & {0}
										\end{bmatrix}\right)}\\
									=& \frac{|q|}{2}\max \{ \|X-Y^*\|,\|Y-X^* \|\}\\
									=& \frac{|q|}{2} \|X-Y^*\|.
								\end{align*}	
								Finally,
								\begin{align*}
									w_q \left(\begin{bmatrix}
										{0} & {X}\\
										{Y} & {0}
									\end{bmatrix}\right) 
									\ge & \frac{|q|}{2}\max\{\|X+Y^*\|,\|X-Y^*\|\} \\
									=& \frac{|q|}{4}(\|X+Y^*\|+\|X-Y^*\|)+ \frac{|q|}{4}|\|X+Y^*\|-\|X-Y^*\||\\
									\ge & \frac{|q|}{4}(\|(X+Y^*) \pm (X-Y^*)\|)+ \frac{|q|}{4}|\|X+Y^*\|-\|X-Y^*\||\\
									\ge & \frac{|q|}{2}\max\{\|X\|, \|Y\| \}+ \frac{|q|}{4}|\|X+Y^*\|-\|X-Y^*\||.
								\end{align*} 
								
							\end{proof}
							
							Taking $X=Y$ in the aforementioned result, we obtain the following corollary.
							\begin{corollary}
								If $X\in \mathcal{B(H)} $ and $q \in \mathcal{D'}$, then we have 
								\begin{equation}\label{e5.8}
									w_q \left(\begin{bmatrix}
										{0} & {X}\\
										{X} & {0}
									\end{bmatrix}\right) 
									\ge \frac{|q|}{2}\|X\|+ \frac{|q|}{2}|\|\mathcal{R}(X)\|-\|\mathcal{I}(X)\||.
								\end{equation}
							\end{corollary}
							\begin{remark}
								For $q=1$, (\ref{e5.8}) gives us a refinement of inequality (\ref{eq2.1}).	 
								Also, if $X^2=0$ then relation (\ref{e5.8}) gives us a significant result as follows
								\begin{equation*}
									\frac{|q|}{2}\|X\| \le w_q\left(\begin{bmatrix}
										{0} & {X}\\
										{X} & {0}
									\end{bmatrix}\right)\le \|X\|.   
								\end{equation*}
								\end{remark}
								In our next result, we extend the following result for $q$-numerical radius.
								\begin{lemma}\cite{abu2015numerical}\label{block}
									Let $T,S \in \mathcal{B(H)}$ are positive operators. Then
									\begin{equation*}
										w \left( \begin{bmatrix}
											0 &T \\
											S &0 
										\end{bmatrix}\right) =\frac{1}{2}\|T+S\|.
									\end{equation*}   
								\end{lemma}
								To prove this, we need the following relation.
								\begin{equation*}
									|q| \sup_{\theta \in \mathbb{R}} \| \mathcal{R}(e^{i \theta}T)\| \le  \sup_{\theta \in \mathbb{R}} w_q( \mathcal{R}(e^{i \theta}T))\\
									\le  \sup_{\theta \in \mathbb{R}} w_q(e^{i \theta}T)\\
									=  w_q(T).
								\end{equation*}
								Hence, for $T \in \mathcal{B(H)}$, we have
								\begin{equation*}
									w_q(T) \ge    |q| \sup_{\theta \in \mathbb{R}}\| \mathcal{R}(e^{i \theta}T)\|.
								\end{equation*}
								
								\begin{theorem}
									Let $X,Y \in \mathcal{B(H)}$, $q \in \mathcal{D'}$ and $0 \le \gamma \le 1$, we have
									\begin{align*}
										\frac{|q|}{2}\sup_{\theta \in \mathbb{R}}\left\lbrace \|e^{i\theta}X+e^{-i\theta}Y^*\|\right\rbrace \le 		w_q\left(\begin{bmatrix}
											{0} & {X}\\
											{Y} & {0}
										\end{bmatrix}\right)\le & \frac{|q|}{2}\| |X|^{2\gamma}+|Y^*|^{2(1-\gamma)}\|^\frac{1}{2}\| |X^*|^{2(1-\gamma)}+|Y|^{2\gamma}\|^\frac{1}{2}\\
										+\sqrt{1-|q|^2}\max\{ \|X\|, \|Y\|\}.	
									\end{align*}    
								\end{theorem}
								\begin{proof}
									From the relation $ w_q(T) \ge    |q| \sup_{\theta \in \mathbb{R}}\| \mathcal{R}(e^{i \theta}T)\|$, we have
									\begin{align*}
										w_q \left(\begin{bmatrix}
											{0} & {X}\\
											{Y} & {0}
										\end{bmatrix}\right) \ge&   |q|\sup_{\theta \in \mathbb{R}} \norm{ \mathcal{R}\left(e^{i \theta} \begin{bmatrix}
												{0} & {X}\\
												{Y} & {0}
											\end{bmatrix}\right)}   \\
										\ge& \frac{|q|}{2}  \sup_{\theta \in \mathbb{R}} \norm{ \begin{bmatrix}
												{0} & {e^{i\theta}X+e^{-i\theta}Y^*}\\
												{e^{i\theta}Y+e^{-i\theta}X^*} & {0}
										\end{bmatrix}}  \\
										= & \frac{|q|}{2}\sup_{\theta \in \mathbb{R}} \|e^{i\theta}X+e^{-i\theta}Y^*\|. 
									\end{align*} 
									Now, to prove the second part, let
									Let $x=\begin{bmatrix}
										{x_1}\\
										{x_2} 
									\end{bmatrix} \in \mathcal{H},$
									and $y=
									\begin{bmatrix}
										{y_1}\\
										{y_2} 
									\end{bmatrix} \in \mathcal{H}$ with $\|x\|=\|y\|=1$ and $\langle x,y \rangle=q$.
									Then we can take $y= \overline{q}x+\sqrt{1-|q|^2}z$, where $z=\begin{bmatrix}
										{z_1}\\
										{z_2} 
									\end{bmatrix} \in \mathcal{H}$, $\|z\|=1$ and $\langle x,z \rangle =0$. Thus,
									$y_1=\overline{q}x_1+\sqrt{1-|q|^2}z_1$ and $y_2=\overline{q}x_2+\sqrt{1-|q|^2}z_2$.
									Hence,
									\begin{align*}
										&\left| \left\langle \begin{bmatrix}
											{0} & {X}\\
											{Y} & {0}
										\end{bmatrix}
										\begin{bmatrix}
											{x_1}\\
											{x_2} 
										\end{bmatrix},
										\begin{bmatrix}
											{y_1}\\
											{y_2} 
										\end{bmatrix}
										\right\rangle \right| \\
										\le& |\langle Xx_2,y_1\rangle|+|\langle Yx_1,y_2\rangle| \\
										=& |\langle Xx_2,\overline{q}x_1+\sqrt{1-|q|^2}z_1\rangle|+|\langle Yx_1,\overline{q}x_2+\sqrt{1-|q|^2}z_2\rangle| \\
										\le & |q| (|\langle Xx_2,x_1 \rangle|+|\langle Yx_1,x_2 \rangle|)+\sqrt{1-|q|^2}(|\langle Xx_2,z_1 \rangle|+|\langle Yx_1,z_2 \rangle|)			
									\end{align*}
									Using Theorem  1\cite{furuta1986simplified} and Cauchy-Schwarz inequality, respectively, we have
									\begin{align*}
										&\left| \left\langle \begin{bmatrix}
											{0} & {X}\\
											{Y} & {0}
										\end{bmatrix}
										\begin{bmatrix}
											{x_1}\\
											{x_2} 
										\end{bmatrix},
										\begin{bmatrix}
											{y_1}\\
											{y_2} 
										\end{bmatrix}
										\right\rangle \right| \\
										\le &|q| (\langle |X|^{2\gamma}x_2,x_2\rangle^\frac{1}{2}\langle |X^*|^{2(1-\gamma)}x_1,x_1\rangle^\frac{1}{2}+\langle |Y|^{2\gamma}x_1,x_1\rangle^\frac{1}{2}\langle |Y^*|^{2(1-\gamma)}x_2,x_2\rangle^\frac{1}{2} )\\
										+& \sqrt{1-|q|^2}(\|X\| \|x_2\| \|z_1\|+\|Y\| \|x_1\| \|z_2\|)\\
										\le&|q|(\langle |X|^{2\gamma}x_2,x_2\rangle+\langle |Y^*|^{2(1-\gamma)}x_2,x_2\rangle)^\frac{1}{2}(\langle |X^*|^{2(1-\gamma)}x_1,x_1\rangle+\langle |Y|^{2\gamma}x_1,x_1\rangle)^\frac{1}{2}\\
										+& \sqrt{1-|q|^2}(\|X\| \|x_2\| \|z_1\|+\|Y\| \|x_1\| \|z_2\|)\\
										\le & |q| \| |X|^{2\gamma}+|Y^*|^{2(1-\gamma)} \|^\frac{1}{2}\| |X^*|^{2(1-\gamma)}+|Y|^{2\gamma}\|^\frac{1}{2}\|x_1\| \|x_2\| \\
										+& \sqrt{1-|q|^2}(\|X\| \|x_2\| \|z_1\|+\|Y\| \|x_1\| \|z_2\|).
									\end{align*}	
									Take $\|x_1\|=\sin(\theta)$, $\|x_2\|=\cos(\theta)$, $\|z_1\|=\cos(\phi)$ and $\|z_2\|=\sin(\phi)$ where $\theta, \phi \in \mathbb{R}$.
									\begin{align*}
										\left| \left\langle \begin{bmatrix}
											{0} & {X}\\
											{Y} & {0}
										\end{bmatrix}
										\begin{bmatrix}
											{x_1}\\
											{x_2} 
										\end{bmatrix},
										\begin{bmatrix}
											{y_1}\\
											{y_2} 
										\end{bmatrix}
										\right\rangle \right|\le&
										|q|\left( \| |X|^{2\gamma}+|Y^*|^{2(1-\gamma)}\|^\frac{1}{2}\| |X^*|^{2(1-\gamma)}+|Y|^{2\gamma}\|^\frac{1}{2}\right)\cos(\theta)\sin(\theta) \\
										+& \sqrt{1-|q|^2}(\|X\| \cos(\theta)\cos(\phi)+\|Y\| \sin(\theta)\sin(\phi))\\
										\le & \frac{|q|\sin(2\theta)}{2}\| |X|^{2\gamma}+|Y^*|^{2(1-\gamma)}\|^\frac{1}{2}\| |X^*|^{2(1-\gamma)}+|Y|^{2\gamma}\|^\frac{1}{2}\\
										+&\sqrt{1-|q|^2}\max\{ \|X\|, \|Y\|\}.	
									\end{align*}	
									Hence,
									\begin{equation*}
										w_q\left(\begin{bmatrix}
											{0} & {X}\\
											{Y} & {0}
										\end{bmatrix}\right)\le \frac{|q|}{2}\| |X|^{2\gamma}+|Y^*|^{2(1-\gamma)}\|^\frac{1}{2}\| |X^*|^{2(1-\gamma)}+|Y|^{2\gamma}\|^\frac{1}{2}+\sqrt{1-|q|^2}\max\{ \|X\|, \|Y\|\}.	
									\end{equation*}
								\end{proof}
								\begin{remark}
									If $X$ and $Y$ are positive operators then, for $\gamma=1/2$, 
									it follows from the above theorem that   
									\begin{equation*}\label{equatione}
										\frac{|q| }{2}\| X+Y\| \le	w_q\left(\begin{bmatrix}
											{0} & {X}\\
											{Y} & {0}
										\end{bmatrix}\right) \le \frac{|q| }{2}\| X+Y\|+\sqrt{1-|q|^2}\max\{ \|X\|,\|Y\|\}.
									\end{equation*}
								\end{remark}
								For $q=1$, it follows $w\left(\begin{bmatrix}
									{0} & {X}\\
									{Y} & {0}
								\end{bmatrix}\right)=\frac{1}{2}\|X+Y\|$, which is mentioned in Lemma \ref{block}. 
								
								In our final result, we assume $X$ and $Y$ both are in $\prod_{s,\alpha}^n $. 
								\begin{theorem}\label{t5.12}
									Let $X,Y \in \prod_{s,\alpha}^n $ and $q \in \mathcal{D'}$,
									\begin{itemize}
										\item [(a)] If $\alpha \ne 0$, then we have 
										\begin{align*}
											w_q \left(\begin{bmatrix}
												{0} & {X}\\
												{Y} & {0}
											\end{bmatrix}\right) \ge& \frac{|q|}{4}\max \{ \|X(1+\cot(\alpha))+Y^*(1-\cot(\alpha))\|,\|X(1-\cot(\alpha))+Y^*(1+\cot(\alpha))\|\}\\
											+&\frac{|q|}{4}\left|\|X+Y^*\|-\cot(\alpha) \|X-Y^*\|\right|.	
										\end{align*}	
										\item [(b)] If $\alpha = 0$, then we have 
										\begin{align*}
											w_q \left(\begin{bmatrix}
												{0} & {X}\\
												{Y} & {0}
											\end{bmatrix}\right) 
											\ge & \frac{|q|}{2}\max\{\|X\|, \|Y\| \}+ \frac{|q|}{4}|\|X+Y\|-\|X-Y\||.
										\end{align*}
									\end{itemize}
									
								\end{theorem}
								\begin{proof}
									\begin{itemize}
										\item [(a)]
										From Theorem \ref{t1.16} and Lemma \ref{il1.6}, the following relations
										$$\|\mathcal{I}(T)\| \le \sin(\alpha)w(T)\le \sin(\alpha)\|T\|
										\le \frac{\tan(\alpha)}{|q|}w_q(T)$$ hold for any $T \in \mathcal{B(H)}$.
										Hence, for $T \in \mathcal{B(H)}$, we have
										\begin{equation}\label{e5.9}
											w_q(T) \ge |q| \|\mathcal{R}(T)\|~~ \text{and}~~ w_q(T) \ge |q| \cot(\alpha)\|\mathcal{I}(T)\|.
										\end{equation}
										Taking $T=\begin{bmatrix}
											{0} & {X}\\
											{Y} & {0}
										\end{bmatrix}$ in inequality \eqref{e5.9}, we have
										\begin{equation*}
											w_q \left(\begin{bmatrix}
												{0} & {X}\\
												{Y} & {0}
											\end{bmatrix}\right) \ge \frac{|q|}{2}\|X+Y^*\| ~~ \text{and} ~~
											w_q \left(\begin{bmatrix}
												{0} & {X}\\
												{Y} & {0}
											\end{bmatrix}\right) \ge \frac{|q|}{2}\cot(\alpha) \|X-Y^*\|. 
										\end{equation*}
										\begin{align*}
											w_q \left(\begin{bmatrix}
												{0} & {X}\\
												{Y} & {0}
											\end{bmatrix}\right) \ge & \frac{|q|}{2}\max \{\|X+Y^*\|,\cot(\alpha) \|X-Y^*\|\}\\
											\ge & \frac{|q|}{4}\left(\|X+Y^*\|+\cot(\alpha) \|X-Y^*\|+\left|\|X+Y^*\|-\cot(\alpha) \|X-Y^*\|\right|\right)\\
											\ge & \frac{|q|}{4}\left(\|(X+Y^*)\pm\cot(\alpha) (X-Y^*)\|+\left|\|X+Y^*\|-\cot(\alpha) \|X-Y^*\|\right|\right).
										\end{align*}
										Thus,
										\begin{align*}
											w_q \left(\begin{bmatrix}
												{0} & {X}\\
												{Y} & {0}
											\end{bmatrix}\right) \ge& \frac{|q|}{4} \Big(\max \{ \|X(1+\cot(\alpha))+Y^*(1-\cot(\alpha))\|,\|X(1-\cot(\alpha))+Y^*(1+\cot(\alpha))\|\}\\
											+&\left|\|X+Y^*\|-\cot(\alpha) \|X-Y^*\|\right|\Big).	
										\end{align*}	
										(b) If $\alpha=0$ then $X$ and $Y$ are positive matrices. We have
										\begin{equation*}
											w_q(T) \ge |q| \|\mathcal{R}(T)\|~~ \text{and}~~ w_q(T) \ge |q|\|\mathcal{I}(T)\|.
										\end{equation*} 
										By using similar calculations as Theorem \ref{thmnew}, we have
										\begin{align*}
											w_q \left(\begin{bmatrix}
												{0} & {X}\\
												{Y} & {0}
											\end{bmatrix}\right) 
											\ge & \frac{|q|}{2}\max\{\|X\|, \|Y^*\| \}+ \frac{|q|}{4}|\|X+Y^*\|-\|X-Y^*\||
										\end{align*} 
										As $X$ and $Y$ are positive so $X=X^*$ and $Y=Y^*$,we have
										\begin{align*}
											w_q \left(\begin{bmatrix}
												{0} & {X}\\
												{Y} & {0}
											\end{bmatrix}\right) 
											\ge & \frac{|q|}{2}\max\{\|X\|, \|Y\| \}+ \frac{|q|}{4}|\|X+Y\|-\|X-Y\||.
										\end{align*} 
									\end{itemize}
								\end{proof}
								
								\begin{remark}
									One notable point is that if $X \in \prod_{s,\alpha}^n$ then $X^* \in \prod_{s,\alpha}^n$. If we take $Y=X^*$, then Theorem \ref{t5.12} gives us 
									\begin{equation*}
										|q| \|X\| \le	w_q \left(\begin{bmatrix}
											{0} & {X}\\
											{X*} & {0}
										\end{bmatrix}\right) \le \|X\|.
									\end{equation*}	
									The lower bound mentioned in the aforementioned inequality is better as compared to the lower bound in Corollary \ref{NEW}(a).
								\end{remark}
								
								\bibliographystyle{unsrt}
								\bibliography{bib_NR_SP}		
								
								\section*{Statements and Declarations:}	
								\subsection*{Competing Interests}	
								The authors have no competing interests.
								\subsection*{Author Contributions}
								All authors contributed equally to this research article.
							\end{document}